\documentclass[11pt,letterpaper]{article}
\usepackage[english]{babel}
\usepackage{amsmath}
\usepackage{amsfonts,mathrsfs}
\usepackage{amssymb}
\usepackage{verbatim}
\usepackage{graphicx}
\usepackage{color}
\usepackage{mathrsfs}

\numberwithin{equation}{section}
\newtheorem{theo}{Theorem}[section]

\newtheorem{prop}[theo]{Proposition}
\newtheorem{lemma}[theo]{Lemma}

\newtheorem{remarks}[theo]{Remarks}

\newtheorem{defn}[theo]{Definition}
\newenvironment{proof}[1][Proof]{\textbf{#1.} }{\ \rule{0.5em}{0.5em}}

\newcommand{\var}{{\rm Var} \mspace{1mu}}

\begin{document}

\begin{center}
{\LARGE Coupling on weighted branching trees}

\vspace{5mm}

{\large Ningyuan Chen \hspace{15mm} Mariana Olvera-Cravioto} \\
{\small Columbia University \hspace{22mm} Columbia University \phantom{hola}}

\vspace{4mm}

\begin{abstract}
This paper considers linear functions constructed on two different weighted branching processes and provides explicit bounds for their Kantorovich-Rubinstein distance in terms of couplings of their corresponding generic branching vectors. Motivated by applications to the analysis of random graphs, we also consider a variation of the weighted branching process where the generic branching vector has a different dependence structure from the usual one. By applying the bounds to sequences of weighted branching processes, we derive sufficient conditions for the convergence in the Kantorovich-Rubinstein distance of linear functions. We focus on the case where the limits are endogenous fixed points of suitable smoothing transformations.
\vspace{5mm}



\end{abstract}

\end{center}

\section{Introduction}

This paper studies one particular solution of the linear stochastic fixed-point equation (SFPE)
\begin{equation} \label{eq:SFPE}
R \stackrel{\mathcal{D}}{=}   \sum_{i=1}^N C_i R_i+Q,
\end{equation}
where $(Q, N, C_1, C_2, \dots)$ is a real-valued vector with $N \in \mathbb{N} \cup \{\infty\}$, and $\{R_i\}_{i \in \mathbb{N}}$ are i.i.d.~random variables having the same distribution as $R$. This distributional equation appears in the probabilistic analysis of algorithms and has been studied in been studied in \cite{Rosler_91, Ros_Rus_01, Fill_Jan_01, Jel_Olv_10, Chen_Litv_Olv_14}, and its homogeneous version $(Q \equiv 0)$  has been studied extensively in the literature of weighted branching processes and multiplicative cascades (see, e.g.,  \cite{Biggins_77, Liu_98, Liu_00, Iksanov_04, Als_Big_Mei_10} and the references therein).

Although it is well known that \eqref{eq:SFPE} has multiple solutions \cite{Als_Big_Mei_10, Alsm_Mein_10a, Alsm_Mein_10b, Iks_Mei_14}, it is often the case that in applications (e.g., \cite{Rosler_91,  Fill_Jan_01, Chen_Litv_Olv_14}) only one of them is relevant. More precisely, we are usually interested in the solution obtained by iterating the SFPE starting from a well-behaved initial condition, which can be explicitly constructed on a weighted branching process (as explained in Section \ref{S.Preliminaries}) and hence is often referred to as a special endogenous solution. We refer the interested reader to \cite{Als_Big_Mei_10, Alsm_Mein_10a, Alsm_Mein_10b, Iks_Mei_14} for a more thorough discussion on the existence of multiple solutions to \eqref{eq:SFPE} and the role of the endogenous solution(s) in their characterization; in particular, the recent work in \cite{Iks_Mei_14} treats the case of real-valued weights $\{C_i\}$ where there may be multiple endogenous solutions.  The focus of this paper is to analyze the ``most attractive" endogenous solution mentioned above both for $Q \equiv 0$ and $P(Q \neq 0) > 0$. In particular, we consider two different weighted branching processes and compare their corresponding special endogenous solutions in the Kantorovich-Rubinstein distance, $d_1$, also known as the Wasserstein distance of order one (see, e.g., \cite{Villani_2009}). Although convergence to the (unique) endogenous solution to \eqref{eq:SFPE} for nonnegative weights has been previously studied in the context of the analysis of divide and conquer algorithms using the Wasserstein distance of order two  \cite{Rosler_91, Ros_Rus_01, Neininger_01}, recent applications to the analysis of information ranking algorithms \cite{Volk_Litv_08, Jel_Olv_10, Chen_Litv_Olv_14} where the variance of $R$ might be infinite, suggest the use of the weaker Kantorovich-Rubinstein distance.

Moreover, motivated by the same applications to the analysis of information ranking algorithms mentioned above, we analyze a variation of the weighted branching process constructed using a generic branching vector of the form $(Q, N, C)$, i.e., where the weight $C$ of a node in the tree is allowed to depend on the node's copy of $(Q, N)$, rather than on its {\em parent's} copy as is the case with a generic branching vector of the form $(Q, N, C_1, C_2, \dots)$. To avoid confusion, we will refer to this variation as a weighted branching tree. Using this weighted branching tree we mimic the construction of the special endogenous solution and provide conditions under which it will be close, in the Kantorovich-Rubinstein distance, to the special endogenous solution to \eqref{eq:SFPE} constructed on a usual weighted branching process.

The main results in this paper provide explicit bounds for the Kantorovich-Rubinstein distance between two random variables constructed according to the representation for the special endogenous solution to \eqref{eq:SFPE}; these bounds are given in terms of the Kantorovich-Rubinstein distance between their generic branching vectors. We then use these bounds to obtain the convergence of a sequence of such random variables in the same distance. We illustrate the main results with applications to the analysis of random graphs and information ranking algorithms.

The remainder of the paper is organized as follows. Section \ref{S.Preliminaries} describes the weighted branching process and its variation, the weighted branching tree. Section \ref{S.Wasserstein} contains a brief exposition of the Kantorovich-Rubinstein distance and some of its main properties. Section \ref{S.Main} contains our main results, with the explicit bounds for the Kantorovich-Rubinstein distance given in Section \ref{S.Coupling}. Finally, Section \ref{S.Applications} contains applications of the main results to the analysis of random graphs and information ranking algorithms. All the proofs in the paper are postponed until Section \ref{S.Proofs}.

\section{Weighted branching processes} \label{S.Preliminaries}

In order to define a weighted branching process we start by letting $\mathbb{N}_+ = \{1, 2, 3, \dots\}$ be the set of positive integers and setting $U = \bigcup_{k=0}^\infty (\mathbb{N}_+)^k$ to be the set of all finite sequences ${\bf i} = (i_1, i_2, \dots, i_n)$, $n\ge 0$, where by convention $\mathbb{N}_+^0 = \{ \emptyset\}$ contains the null sequence $\emptyset$. To ease the exposition, for a sequence ${\bf i} = (i_1, i_2, \dots, i_k) \in U$ we write ${\bf i}|n = (i_1, i_2, \dots, i_n)$, provided $k \geq n$, and ${\bf i}|0 = \emptyset$ to denote the index truncation at level $n$, $n \geq 0$. Also, for ${\bf i} \in A_1$ we simply use the notation ${\bf i} = i_1$, that is, without the parenthesis. Similarly, for ${\bf i} = (i_1, \dots, i_n)$ we will use $({\bf i}, j) = (i_1,\dots, i_n, j)$ to denote the index concatenation operation, if ${\bf i} = \emptyset$, then $({\bf i}, j) = j$.

Next, let $(Q, N, C_1, C_2, \dots)$ be a real-valued vector with $N \in \mathbb{N} \cup \{\infty\}$. We will refer to this vector as the generic branching vector. Now let $\{ (Q_{\bf i}, N_{\bf i}, C_{({\bf i}, 1)}, C_{({\bf i}, 2)}, \dots ) \}_{{\bf i} \in U}$ be a sequence of i.i.d.~copies of the generic branching vector. To construct a weighted branching process we start by defining a tree as follows: let $A_0 = \{ \emptyset\}$ denote the root of the tree, and define the $n$th generation according to the recursion
$$A_n = \{ ({\bf i}, i_n) \in U: {\bf i} \in A_{n-1}, 1 \leq i_n \leq N_{\bf i} \}, \quad n \geq 1.$$
Now, assign to each node ${\bf i}$ in the tree a weight $\Pi_{\bf i}$ according to the recursion
$$\Pi_\emptyset \equiv 1, \qquad \Pi_{({\bf i}, i_n)} = C_{({\bf i}, i_n)} \Pi_{\bf i}, \qquad n \geq 1, $$
see Figure \ref{F.Tree}. Note that the tree's structure, disregarding the weights, is that of a Galton-Watson process with offspring distribution $f(k) = P(N = k)$, provided $P(N < \infty) = 1$.

\begin{figure}[t]
\centering
\begin{picture}(350,110)(0,0)
\put(0,0){\includegraphics[scale = 0.75, bb = 30 560 510 695, clip]{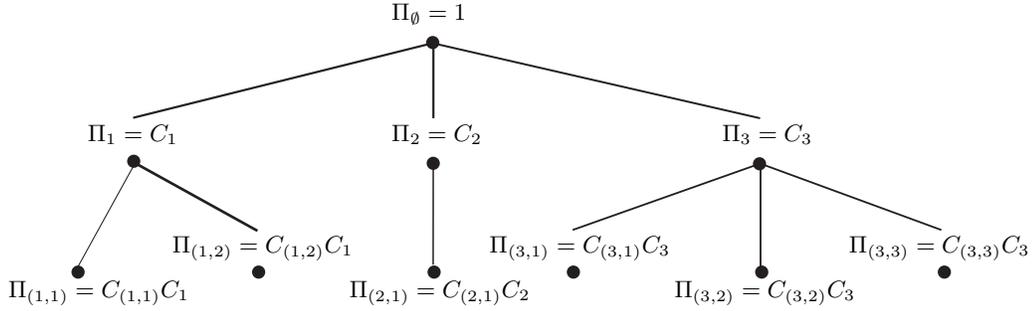}}
\put(145,105){\footnotesize $\Pi_\emptyset = 1$}
\put(30,59){\footnotesize $\Pi_{1} = C_1$}
\put(145,59){\footnotesize $\Pi_{2} = C_2$}
\put(270,59){\footnotesize $\Pi_{3} = C_3$}
\put(0,0){\footnotesize $\Pi_{(1,1)} = C_{(1,1)} C_1$}
\put(62,17){\footnotesize $\Pi_{(1,2)} = C_{(1,2)} C_1$}
\put(129,0){\footnotesize $\Pi_{(2,1)} = C_{(2,1)} C_2$}
\put(182,17){\footnotesize $\Pi_{(3,1)} = C_{(3,1)} C_3$}
\put(252,0){\footnotesize $\Pi_{(3,2)} = C_{(3,2)} C_3$}
\put(318,17){\footnotesize $\Pi_{(3,3)} = C_{(3,3)} C_3$}
\end{picture}
\caption{Weighted branching process}\label{F.Tree}
\end{figure}

Using the same notation described above, consider now constructing this process using a generic branching vector of the form $(Q, N, C)$, with $N \in \mathbb{N}$, and a sequence of i.i.d.~copies $\{(Q_{\bf i}, N_{\bf i}, C_{\bf i})\}_{{\bf i} \in U}$. As mentioned earlier, we will refer to this construction as a weighted branching tree. The difference lies in the dependence structure that now governs the nodes in the tree, since whereas in a usual weighted branching process the weight $C_{\bf i}$ of node ${\bf i}$ is independent of $(Q_{\bf i}, N_{\bf i})$, in a weighted branching tree it may not be. Another important observation is that in a weighted branching tree the weights $\{C_{\bf i}\}_{{\bf i} \in U}$ are i.i.d.~random variables, unlike in a weighted branching process where the weights of ``sibling" nodes are arbitrarily dependent and not necessarily identically distributed. It follows from these observations that when $C$ is independent of $(Q,N)$, the corresponding weighted branching tree is a special case of a weighted branching process.

We will now explain how to construct the special endogenous solution to the linear SFPE \eqref{eq:SFPE} using a weighted branching process.

\subsection{The special endogenous solution to the linear SFPE} \label{S.EndogenousSol}

For a weighted branching process with generic branching vector $(Q, N, C_1, C_2, \dots)$, define the processes $\{ W^{(j)}: j \geq 0\}$ and $\{ R^{(k)} : k \geq 0\}$ as follows:
\begin{align}
W^{(0)} &= Q_0, \qquad W^{(j)} = \sum_{{\bf i} \in A_j} Q_{\bf i} \Pi_{\bf i}, \quad j \geq 1, \label{eq:Wprocess} \\
R^{(k)} &= \sum_{j=0}^k W^{(j)} = \sum_{j=0}^k \sum_{{\bf i} \in A_j} Q_{\bf i} \Pi_{\bf i}, \quad k \geq 0. \label{eq:Rprocess}
\end{align}

By focusing on the branching vector belonging to the root node, i.e., $(Q_\emptyset, N_\emptyset, C_1, C_2, \dots)$ we can see that the processes $\{W^{(j)}\}$ and $\{ R^{(k)} \}$ satisfy the distributional equations
\begin{equation} \label{eq:Homogeneous}
W^{(j)} = \sum_{r=1}^{N_\emptyset} C_r \left( \sum_{(r,{\bf i}) \in A_j} Q_{(r,{\bf i})} \Pi_{(r,{\bf i})}/ C_r   \right) \stackrel{\mathcal{D}}{=} \sum_{r=1}^{N} C_r W^{(j-1)}_r, \qquad j \geq 1,
\end{equation}
and
\begin{equation} \label{eq:NonHomogeneous}
R^{(k)} = \sum_{r=1}^{N_\emptyset} C_r \left( \sum_{j=1}^k \sum_{(r,{\bf i}) \in A_j} Q_{(r,{\bf i})} \Pi_{(r,{\bf i})}/ C_r  \right) + Q_\emptyset \stackrel{\mathcal{D}}{=} \sum_{r=1}^{N} C_r R^{(k-1)}_r + Q, \qquad k \geq 1,
\end{equation}
where $W^{(j-1)}_r$ are i.i.d.~copies of $W^{(j-1)}$ and $R^{(k-1)}_r$ are i.i.d.~copies of $R^{(k-1)}$, all independent of $(Q, N, C_1, C_2, \dots)$. Here and throughout the paper the convention is that $\Pi_{(r,{\bf i})}/C_r \equiv 1$ if $C_r = 0$.

For the homogeneous case ($Q \equiv 0$ in \eqref{eq:SFPE}), assume the weights $\{C_i\}$ are nonnegative and redefine the $\{W^{(j)}\}$ process as
$$W^{(0)} = 1, \qquad W^{(j)} = \sum_{{\bf i} \in A_j} \Pi_{\bf i}, \quad j \geq 1.$$
In this case, and provided $\rho = E[ \sum_{i=1}^N C_i ] < \infty$, the process $M^{(j)} = W^{(j)}/ \rho^j$, $j \geq 0$, defines a nonnegative martingale. It follows that $M^{(j)}$ converges almost surely, as $j \to \infty$, to a finite limit $W$ with $E[ W ] \leq 1$. Taking the limit as $j \to \infty$ in \eqref{eq:Homogeneous} then gives that $W$ satisfies
$$W \stackrel{\mathcal{D}}{=} \sum_{r=1}^N \frac{C_r}{\rho} \, W_r \triangleq \sum_{r=1}^N C_r' \, W_r,$$
where the $\{W_r\}$ are i.i.d.~copies of $W$, independent of $(N, C_1, C_2, \dots)$. Hence, $W$ is a solution to the homogeneous version of \eqref{eq:SFPE} with the generic branching vector $(N, C_1', C_2', \dots)$.

For the nonhomogeneous case $(P(Q \neq 0) > 0$), one can argue, as was done in \cite{Jel_Olv_12b}, that provided $E[ \sum_{i=1}^N |C_i|^\beta ] < 1$ and $E[|Q|^\beta] < \infty$ for some $0 < \beta \leq 1$, then the random variable $R^{(k)}$ converges almost surely, as $k \to \infty$, to a finite limit $R$. Taking the limit as $k \to \infty$ in \eqref{eq:NonHomogeneous} gives that $R$ is a solution to \eqref{eq:SFPE}. We refer to the random variables $W$ and $R$ described above as the special endogenous solutions to \eqref{eq:SFPE} in the homogeneous and nonhomogeneous cases, respectively. We point out that in the case of nonnegative weights, $W$ and $R$ are the unique endogenous solutions to \eqref{eq:SFPE}, whereas in the real-valued case there can be other endogenous solutions, i.e., that can be explicitly constructed using a weighted branching process (see \cite{Iks_Mei_14} for more details).

\section{The Kantorovich-Rubinstein distance} \label{S.Wasserstein}

Before proceeding to the main results in the paper we give a brief description of the Kantorovich-Rubinstein distance. This distance on the space of probability measures is also known as the minimal $l_1$ metric or the Wasserstein distance of order one. For the purposes of this paper, we consider the vector space of infinite real sequences $\mathbb{R}^\infty$ having finite $l_1$ norm, i.e., ${\bf x} \in \mathbb{R}^\infty$ such that
$$\| {\bf x} \|_1 = \sum_{i=1}^\infty |x_i| < \infty.$$
On some occasions, which will become clear from the context, we will work with elements of $\mathbb{R}^d$ instead. The norm $\| {\bf x} \|_1$ will always refer to the corresponding $l_1$ norm.

\begin{defn} \label{D.Wasserstein}
    Let $M(\mu, \nu)$ denote the set of joint probability measures on $\mathcal S\times \mathcal S$ ($\mathcal S=\mathbb R^d$ or $\mathbb R^{\infty}$) with marginals $\mu$ and $\nu$. Then, the Kantorovich-Rubinstein distance between $\mu$ and $\nu$ is given by
$$d_1(\mu, \nu) = \inf_{\pi \in M(\mu, \nu)} \int_{\mathcal S\times \mathcal S} \| {\bf x} - {\bf y} \|_1 \, d \pi({\bf x}, {\bf y}).$$
\end{defn}

We point out that $d_1$ is only strictly speaking a distance when restricted to the subset of probability measures
$$\mathscr{P}_1(\mathcal{S}) \triangleq \left\{ \mu \in \mathscr{P}(\mathcal S): \int_{\mathcal S} \| {\bf x} \|_1 \, d\mu( {\bf x}) < \infty \right\},$$
where $\mathscr{P}(\mathcal S)$ is the set of Borel probability measures on $\mathcal S$. We refer the interested reader to \cite{Villani_2009} for a thorough treatment of this distance, since Definition \ref{D.Wasserstein} is only a special case.

Any construction on the same probability space of the joint vector $({\bf X}, {\bf Y})$, where ${\bf X}$ has distribution $\mu$ and ${\bf Y}$ has distribution $\nu$, is called a {\em coupling} of $\mu$ and $\nu$. In this notation we can rewrite $d_1$ as
$$d_1(\mu, \nu) = \inf_{{\bf X}, {\bf Y}} E\left[ \| {\bf X} - {\bf Y} \|_1 \right],$$
where the infimum is taken over all couplings of $\mu$ and $\nu$.

It is well known that $d_1$ is a metric on $\mathscr{P}_1(\mathcal{S})$ and that the infimum is attained, or equivalently, that an optimal coupling $({\bf X}, {\bf Y})$ such that
$$d_1(\mu, \nu) = E\left[ \| {\bf X} - {\bf Y} \|_1 \right]$$
always exists (see, e.g., \cite{Villani_2009}, Theorem 4.1 or \cite{Bick_Free_81}, Lemma 8.1). This optimal coupling, nonetheless, is not in general explicitly available. One noteworthy exception is when $\mu$ and $\nu$ are probability measures on the real line, in which case we have that
$$d_1(\mu, \nu) = \int_{0}^1 | F^{-1}(u) - G^{-1}(u) | du = \int_{-\infty}^\infty | F(x) - G(x) | dx,$$
where $F$ and $G$ are the cumulative distribution functions of $\mu$ and $\nu$, respectively, and $f^{-1}(t) = \inf\{ x \in \mathbb{R}: f(x) \geq t\}$ denotes the pseudo-inverse of $f$. It follows that the optimal coupling is given by $(X, Y) = (F^{-1}(U), G^{-1}(U))$ for $U$ uniformly distributed in $[0, 1]$.

Another important property of the Kantorovich-Rubinstein distance is that the convergence in $d_1$ to a limit $\mu \in \mathscr{P}_1(\mathcal{S})$ is equivalent to weak convergence plus convergence of the first moments. Furthermore, it satisfies the useful {\bf duality formula}:
$$d_1(\mu, \nu) = \sup_{\| \psi \|_{\text{Lip}} \leq 1} \left\{ \int_{\mathcal S} \psi({\bf x}) d\mu({\bf x)} - \int_{\mathcal S} \psi({\bf x}) d\nu({\bf x)} \right\}$$
for all $\mu, \nu \in \mathscr{P}_1(\mathcal{S})$, where the supremum is taken over all Lipschitz continuous functions $\psi: \mathcal S \to \mathbb{R}$ with Lipschitz constant one (see Remark 6.5 in \cite{Villani_2009}).

\section{Main Results}\label{S.Main}

The paper contains two sets of results; the first one provides explicit bounds for the Kantorovich-Rubinstein distance between two versions of the processes $\{W^{(j)}: j \geq 0\}$ (as defined by \eqref{eq:Wprocess}) constructed on weighted branching processes, respectively weighted branching trees, using different generic branching vectors. These bounds are given in terms of the Kantorovich-Rubinstein distance between the two generic branching vectors. The second set of results apply the explicit bounds to a sequence of processes $\{ W^{(n,j)}: j \geq 0\}$ and $\{ R^{(n,k)}: k \geq 0\}$ for $n\geq 1$, to obtain the convergence in the Kantorovich-Rubinstein distance to the special endogenous solution to \eqref{eq:SFPE} constructed on a limiting weighted branching process.

\subsection{Bounds for the Kantorovich-Rubinstein distance} \label{S.Coupling}

Let $\{W^{(j)}: j \geq 0\}$ and $\{ \hat W^{(j)}: j \geq 0\}$ be defined according to \eqref{eq:Wprocess} on two different weighted branching processes using the generic branching vectors $(Q, N, C_1, C_2, \dots)$ and $(\hat Q, \hat N, \hat C_1, \hat C_2, \dots)$, respectively. As our result will show, it is enough to consider generic branching vectors of the form $(Q, B_1, B_2, \dots)$ and $(\hat Q, \hat B_1, \hat B_2, \dots)$ where $B_i = C_i 1(N \geq i)$ and $\hat B_i = \hat C_i 1(\hat N \geq i)$ for all $i \in \mathbb{N}_+$. Let $\mu$ denote the probability measure of $(Q, B_1, B_2, \dots)$ and let $\hat \mu$ denote the probability measure of $(\hat Q, \hat B_1, \hat B_2, \dots)$. Using $\mathcal S=\mathbb R^{\infty}$, we assume throughout the paper that
\begin{equation} \label{eq:FiniteNorm}
    \int_{\mathbb R^{\infty}} \| {\bf x} \|_1 d\mu ({\bf x}) =E\left[ |Q|+ \sum_{i=1}^{\infty}|B_i| \right]< \infty \quad \text{and} \quad \int_{\mathbb R^{\infty}} \| {\bf x} \|_1 d\hat \mu ({\bf x})=E\left[ |\hat Q| + \sum_{i=1}^{\infty}|\hat{B}_i| \right] < \infty.
\end{equation}

To construct the two processes on the same probability space, let $\pi$ denote any coupling of $\mu$ and $\hat \mu$ and let $\{ (Q_{\bf i}, B_{({\bf i},1)}, B_{({\bf i},2)}, \dots, \hat Q_{\bf i}, \hat B_{({\bf i},1)}, \hat B_{({\bf i},2)}, \dots) \}_{{\bf i} \in U}$ be a sequence of i.i.d.~random vectors distributed according to $\pi$. Then, use the vectors $\{ (Q_{\bf i}, B_{({\bf i},1)}, B_{({\bf i},2)}, \dots) \}_{{\bf i} \in U}$ to construct $\{W^{(j)}: j \geq 0\}$, as described in Section \ref{S.Preliminaries}, and the vectors $\{ ( \hat Q_{\bf i}, \hat B_{({\bf i},1)}, \hat B_{({\bf i},2)}, \dots) \}_{{\bf i} \in U}$ to construct $\{\hat W^{(j)}: j \geq 0\}$. Our first result is stated below. We use the convention that $\sum_{i=a}^b x_i \equiv 0$ if $b < a$, and the notation $E_\pi[ \cdot]$ to denote the expectation taken with respect to the coupling $\pi$; we also use $x \wedge y$ and $x \vee y$ to denote the minimum and the maximum of $x$ and $y$, respectively, and $x^+ = \max\{0, x\}$.

\begin{prop} \label{P.CouplingWBP}
For any coupling $\pi$ of $\mu$ and $\hat \mu$, and any $j \geq 0$,
\begin{equation*}
    E\left[ \left| \hat W^{(j)} - W^{(j)} \right| \right]\le\left(\hat\rho^{j} +  E[|Q|] \sum_{t=0}^{j-1}\rho^t\hat\rho^{j-1-t}\right) \mathcal{E},
\end{equation*}
where $\rho = E[ \sum_{i=1}^N |C_i| ]$, $\hat \rho = E[ \sum_{i=1}^{\hat N} |\hat C_i| ]$ and $\mathcal{E} = E_\pi [ |\hat Q - Q| + \sum_{i=1}^\infty |\hat B_i - B_i| ].$
\end{prop}

We point out that the bound provided by Proposition \ref{P.CouplingWBP} is also a bound for the Kantorovich-Rubinstein distance between $\hat W^{(j)}$ and $W^{(j)}$, and if we take $\pi$ to be the optimal coupling of $\mu$ and $\hat \mu$ then we have $\mathcal{E} = d_1(\hat\mu, \mu)$. It is also worth mentioning that if we let $\nu$ and $\hat \nu$ be the probability measures of $(Q, N, C_1, C_2, \dots)$ and $(\hat Q, \hat N, \hat C_1, \hat C_2, \dots)$, respectively, and assume that $E[N + \hat N] < \infty$, then $d_1(\mu, \hat \mu)$ can be small even if $d_1(\nu, \hat \nu)$ is not. This is due to the observation that, in general, large disagreements between $C_r$ and $\hat C_r$ for values of $r$ for which $P(N > r)$ and $P(\hat N > r)$ are negligible do not affect $d_1(\mu, \hat \mu)$, whereas they do adversely affect $d_1(\nu, \hat \nu)$.

Our next result provides a similar bound for the case when $\hat W^{(j)}$ and $W^{(j)}$ are constructed on weighted branching trees using the generic branching vectors $(\hat Q, \hat N, \hat C)$ and $(Q, N, C)$, respectively. As before, let $\hat \nu$ and $\nu$ denote the probability measures of $(\hat Q, \hat N, \hat C)$ and $(Q, N, C)$. We allow the coupling used for the root nodes to be different than all other nodes, i.e., the two trees are constructed using the sequence of i.i.d.~vectors $\{ (Q_{\bf i}, C_{\bf i}, N_{\bf i}, \hat Q_{\bf i}, \hat C_{\bf i}, \hat N_{\bf i})\}_{{\bf i} \in U, {\bf i} \neq \emptyset}$ distributed according to a coupling $\pi$ of $\nu$ and $\hat \nu$, while $(Q_{\emptyset}, N_{\emptyset}, \hat Q_{\emptyset}, \hat N_{\emptyset})$ is independent of the previous sequence and is distributed according to a coupling $\pi^*$ of $\nu^*$ and $\hat \nu^*$, where $\nu^*$ is the probability measure of $(Q, N)$ and $\hat \nu^*$ is that of $(\hat Q, \hat N)$. We have ignored $C_\emptyset$ and $\hat C_{\emptyset}$ since they do not appear in the definitions of $W^{(j)}$ and $\hat W^{(j)}$.

\begin{prop} \label{P.CouplingWBT}
For any coupling $\pi$ of $\nu$ and $\hat \nu$ and any coupling $\pi^*$ of $\nu^*$ and $\hat \nu^*$,
$${E} \left[ \left| \hat W^{(0)} - W^{(0)} \right| \right] \leq  \mathcal{E}^*$$
and for $j \geq 1$,
$${E} \left[ \left| \hat W^{(j)} - W^{(j)} \right| \right] \leq  \left( E[\hat N] \vee \frac{E[N] E[|CQ|]}{\rho} \right) \left( \sum_{t=0}^{j-1} \hat \rho^t \rho^{j-1-t} \right) \mathcal{E} + E[|Q|] \hat \rho^{j-1} \mathcal{E}^*,$$
where $\rho = E[ N |C|]$, $\hat \rho = E[ \hat N |\hat C|]$,
\begin{align*}
\mathcal{E}^* &= E_{\pi^*} \left[ | \hat Q - Q| + |\hat N - N| \right] \quad \text{and} \quad \mathcal{E} = E_\pi \left[ |\hat C \hat Q - CQ| +  \sum_{i=1}^\infty |\hat C 1(\hat N \geq i) - C 1 (N \geq i) | \right] .
\end{align*}
\end{prop}

\subsection{Convergence to the special endogenous solution} \label{S.Convergence}

Our second set of results considers a sequence of weighted branching processes (respectively, weighted branching trees), each constructed using a generic branching vector having probability measure $\nu_n$, $n \geq 1$. In other words, for weighted branching processes, $\nu_n$ is the probability measure of a vector of the form $(Q^{(n)}, N^{(n)}, C^{(n)}_1, C^{(n)}_2, \dots)$, while for weighted branching trees it corresponds to a vector of the form $(Q^{(n)}, N^{(n)}, C^{(n)})$. On each of them we define the processes $\{ W^{(n,j)}: j \geq 0\}$ and $\{ R^{(n,k)}: k \geq 0\}$ according to \eqref{eq:Wprocess} and \eqref{eq:Rprocess}, and we are interested in providing conditions under which $W^{(n,j)}$ (suitably scaled) and $R^{(n,k)}$ will converge, as $n, j, k$ go to infinity, to the special endogenous solution of a linear SFPE of the form in \eqref{eq:SFPE}.

The main conditions for the convergence we seek will be in terms of the sequence of probability measures $\{\mu_n\}_{n \geq 1}$, where $\mu_n$ is the probability measure of the vector
$$(Q^{(n)}, C_1^{(n)} 1(N^{(n)} \geq 1),  C_2^{(n)} 1(N^{(n)} \geq 2), \dots)$$
for weighted branching processes, and of
$$(C^{(n)} Q^{(n)}, C^{(n)} 1(N^{(n)} \geq 1),  C^{(n)} 1(N^{(n)} \geq 2), \dots)$$
for weighted branching trees.

In both cases, we assume that there exists a probability measure $\mu$ such that $d_1(\mu_n, \mu) \to 0$. We point out that for a weighted branching process, $\mu$ is always the probability measure of a generic branching vector, since each of the $\mu_n$ is. However, this is not necessarily the case for a weighted branching tree, and we need to further assume that there exist probability measures $\eta_1$ on $\mathbb{R}$ and $\eta_2$ on $\mathbb{R} \times \{0,1\}^\infty$ such that
$$\int_{\mathbb{R}^\infty} h({\bf x}) \mu(d{\bf x}) = \int_{\mathbb{R}\times\{0,1\}^\infty} \int_{\mathbb{R}} h(yx_1, yx_2, yx_3, \dots) \eta_1(dy) \eta_2(d{\bf x})$$
for all functions $h:\mathbb{R}^\infty \to \mathbb{R}$. We can then identify $\eta_1$ with the probability distribution of $C$
and $\eta_2$ with the probability distribution of the vector $(Q, 1(N>1), 1(N> 2), \dots)$, which fully determines $(Q,N)$. With this interpretation, the limiting measure $\mu$ defines a weighted branching process with a generic branching vector of the form $(Q, N, C_1, C_2, \dots)$ with the $\{C_i\}_{i \geq 1}$ i.i.d.~and independent of $(Q, N)$; condition \eqref{eq:FiniteNorm} implies that $E[N] < \infty$.

We refer to the case where we analyze a sequence of weighted branching processes as {\it Case~1}, and to the case where we analyze a sequence of weighted branching trees as {\it Case~2}. For {\it Case~2}, in addition to the measure $\mu_n$ defined above, we define $\nu_n^*$ to be the probability measure of the vector $(Q^{(n)}, N^{(n)})$ and $\nu^*$ to be the probability measure of $(Q, N)$. The symbol $\Rightarrow$ denotes convergence in distribution and $\stackrel{d_1}{\longrightarrow}$ denotes convergence in the Kantorovich-Rubinstein distance.

\begin{theo} \label{T.MainHomo}
Define the processes $\{ W^{(n,j)}: j \geq 0\}$, $n \geq 1$, and $\{W^{(j)}: j \geq 0\}$ according to \eqref{eq:Wprocess}. Suppose that as $n \to \infty$,
$$d_1(\mu_n,\mu) \to 0 \quad \text{({\it Case 1})} \quad \text{or}  \quad d_1(\nu_n^*,\nu^*) + d_1(\mu_n,\mu) \to 0 \quad \text{({\it Case 2})}.$$
Then, for any fixed $j\in\mathbb N$
$$ W^{(n,j)} \stackrel{d_1}{\longrightarrow} W^{(j)}, \qquad n \to \infty.$$
Moreover, if $Q^{(n)} = Q \equiv 1$, and $C_j^{(n)}, C_j$ are nonnegative for all $n$ and $j$, then for any $j_n \in \mathbb{N}$ such that $j_n \to \infty$ and
$$ j_n \, d_1(\mu_n,\mu) \to 0 \quad \text{({\it Case 1})} \quad \text{or} \quad d_1(\nu_n^*, \nu^*)+  j_n \, d_1(\mu_n,\mu) \to 0 \quad \text{({\it Case 2})},$$
as $n \to \infty$, we have
$$\frac{W^{(n, j_n)}}{\rho_n^{j_n}} \Rightarrow \mathcal{W} \qquad \text{and} \qquad \frac{W^{(n, j_n)}}{\rho^{j_n}} \Rightarrow \mathcal{W},$$
where $\mathcal{W}$ is the a.s. limit of $W^{(j)}/\rho^j$ as $j \to \infty$.
\end{theo}

As pointed out in Section \ref{S.EndogenousSol}, $\mathcal{W}$ is the unique endogenous solution to the SFPE
$$\mathcal{W} \stackrel{\mathcal{D}}{=} \sum_{i=1}^N \frac{C_i}{\rho} \, \mathcal{W}_i,$$
where the $\{ \mathcal{W}_i\}$ are i.i.d.~copies of $\mathcal{W}$, independent of $(N, C_1, C_2, \dots)$. See \cite{Lyons_97, Liu_98, Als_Iks_09} for conditions on when the random variable $\mathcal W$, which satisfies $E[\mathcal W] \leq 1$, is non-trivial, as well as characterizations of its tail behavior. Furthermore, when $E[\mathcal{W}] = 1$ we can replace the convergence in distribution with convergence in the Kantorovich-Rubinstein distance, i.e.,
$$\frac{W^{(n, j_n)}}{\rho_n^{j_n}} \stackrel{d_1}{\longrightarrow} \mathcal{W} \qquad \text{and} \qquad \frac{W^{(n, j_n)}}{\rho^{j_n}} \stackrel{d_1}{\longrightarrow} \mathcal{W}, \qquad n \to \infty.$$

We now give a similar result for the nonhomogeneous equation.

\begin{theo} \label{T.MainNonHomo}
Define the processes $\{ R^{(n,k)}: k \geq 0\}$, $n \geq 1$, and $\{R^{(k)}: k \geq 0\}$ according to \eqref{eq:Rprocess}. Suppose that as $n \to \infty$,
$$d_1(\mu_n,\mu) \to 0 \quad \text{({\it Case 1})} \quad \text{or}  \quad d_1(\nu_n^*,\nu^*) + d_1(\mu_n,\mu) \to 0 \quad \text{({\it Case 2})}.$$
Then, for any fixed $k \in \mathbb{N}$,
$$ R^{(n,k)} \stackrel{d_1}{\longrightarrow} R^{(k)}, \qquad n \to \infty.$$
Moreover, if $\rho < 1$, then for any $k_n \in \mathbb{N}$ such that $k_n \to \infty$ as $n \to \infty$, we have
$$R^{(n,k_n)} \stackrel{d_1}{\longrightarrow} R, \qquad n \to \infty, $$
where $R = \sum_{k=0}^\infty \sum_{{\bf i} \in A_k} \Pi_{\bf i} Q_{\bf i}$ is the a.s. limit of $R^{(k)}$ as $k \to \infty$.
\end{theo}

In the statement of the theorem, provided $\rho<1$, $R$ is the unique endogenous solution to the SFPE
\begin{equation} \label{eq:SFPE2}
R \stackrel{\mathcal{D}}{=} \sum_{i=1}^N C_i R_i + Q,
\end{equation}
where the $\{R_i\}$ are i.i.d.~copies of $R$, independent of $(Q, N, C_1, C_2, \dots)$. Moreover, the asymptotic behavior of $P( R > x)$ as $x \to \infty$ can be described for several different assumptions on the generic vector $(Q, N, C_1, C_2, \dots)$. We refer the reader to \cite{Jel_Olv_12b} and \cite{Olvera_12} for the precise set of theorems.

Note that in {\it Case 1}, the convergence of $R^{(n,k)}$ as $k \to \infty$ for a fixed $n$ is guaranteed whenever $E[ \sum_{i=1}^{N^{(n)}} | C_i^{(n)}|^\beta ] < 1$ for some $0 < \beta \leq 1$ (see Lemma 4.1 in \cite{Jel_Olv_12b}), and its limit, $R^{(n)}$ would be the unique endogenous solution to
\begin{equation} \label{eq:SFPEn}
R^{(n)} \stackrel{\mathcal{D}}{=} \sum_{i=1}^{N^{(n)}} C_i^{(n)} R_i^{(n)} + Q^{(n)}.
\end{equation}
For {\it Case 2}, on the other hand, an adaptation of the proof of Lemma 4.1 in \cite{Jel_Olv_12b} would give that $R^{(n,k)}$ converges a.s. to
$$R^{(n)} = \sum_{j=0}^\infty W^{(n,j)},$$
as $k \to \infty$, with $R^{(n)}$ finite a.s., provided $E[ N^{(n)} | C^{(n)} |^\beta ] < 1$ for some $0 < \beta \leq 1$. However, this random variable $R^{(n)}$ would not necessarily have the interpretation of being a solution to \eqref{eq:SFPEn}.

We end this section with a result for the weighted branching tree setting that states that $d_1(\mu_n,\mu)$ converges to zero whenever $d_1(\nu_n,\nu)$ and the moments of $Q^{(n)} C^{(n)}$ and $N^{(n)} C^{(n)}$ do.

\begin{lemma} \label{L.FinalConditionsWBT}
For {\it Case 2}, suppose that as $n \to \infty$, $d_1(\nu_n, \nu) \to 0$, $E[|C^{(n)} Q^{(n)}|] \to E[|CQ|]$ and $E[|C^{(n)}| N^{(n)} ] \to E[|C|N]$. Then,
$$d_1(\mu_n, \mu) \to 0, \qquad n \to \infty.$$
\end{lemma}

\section{Applications} \label{S.Applications}

As mentioned in the introduction, our interest in analyzing the processes $\{\hat W^{(j)}:j\geq0\}$ and $\{ \hat R^{(k)}: k \geq 0\}$ when they are constructed on a weighted branching tree, rather than a weighted branching process, comes from applications to the analysis of random graphs and information ranking algorithms. This section provides two examples in which an application of the main results in Section \ref{S.Main} lead to the special endogenous solution to \eqref{eq:SFPE}.

\subsection{Ranking algorithms on a directed configuration network}

Our first example is related to the analysis of ranking algorithms on directed graphs. In particular, we are interested in studying the distribution of the ranks produced by spectral ranking algorithms, e.g., Google's PageRank, used to rank webpages on the World Wide Web. More precisely, the recent work in \cite{Chen_Litv_Olv_14} considers a sequence of random graphs constructed according to the directed configuration model \cite{Chen_Olv_13} and shows that the rank of a randomly chosen node can be coupled with a random variable $R^{(n,k_n)}$ (of the form in \eqref{eq:Rprocess} and built on a weighted branching tree) where the $n$ refers to the number of nodes in the graph. An application of a version of Theorem \ref{T.MainNonHomo} then leads to the rank of a randomly chosen node having a representation in terms of the special endogenous solution  to \eqref{eq:SFPE}, as defined in Section \ref{S.EndogenousSol}, as the number of nodes in the graph grows to infinity.

Before stating the precise version of Theorem~\ref{T.MainNonHomo} that is needed in this application it will be helpful to give some details about the configuration model. The configuration or pairing model (see, e.g., \cite{HofstadRG, Bollobas_2001}), produces a random graph from a given degree sequence by assigning to each node a number of half-edges equal to its degree and then randomly pairing the half-edges to form a graph. Similarly, the directed configuration model generates a directed graph from a given bi-degree sequence (sequence of in-degrees and out-degrees). In both cases, the resulting graph, conditional on it not having self-loops or multiple edges, is a graph uniformly chosen at random from all simple graphs having that degree (bi-degree) sequence. We observe that given the degree sequence(s), the randomness in the graph comes from the pairing process, so it makes sense that in applications one often works on the conditional probability space given the degree sequence(s).

It is due to this last observation that in order to obtain the convergence of the rank of a randomly chosen node on a directed configuration model one needs to apply Theorem~\ref{T.MainNonHomo} conditionally on the sigma-algebra generated by the bi-degree sequence. Other random graph models, e.g., the generalized random graph \cite{HofstadRG}, require conditioning on a ``weight" sequence. Moreover, the analysis of problems related to the configuration model in general (directed or undirected) often relies on a coupling with a weighted branching tree where the root node has a different distribution, hence the need to further tailor the theorem.

In order to state a suitable theorem for the analysis of spectral algorithms on random graphs, we first need to introduce some additional notation. We consider a sequence of sigma-algebras generated by a finite set of random variables/vectors (e.g., the degree sequences). Next, for each $n \geq 1$, and conditionally on $\mathscr{F}_n$, we construct a weighted branching tree using the generic branching vector $(Q^{(n)}, N^{(n)}, C^{(n)})$, whose (conditional) probability measure we denote by $\nu_n$. Moreover, we allow the root branching vector, $(Q^{(n)}_\emptyset, N^{(n)}_\emptyset)$ to have a different (conditional) distribution, say having a probability measure $\nu_n^*$. In other words, the $n$th weighted branching tree is constructed, conditionally on $\mathscr{F}_n$, using the sequence of conditionally i.i.d.~vectors $\{ (Q^{(n)}_{\bf i}, N^{(n)}_{\bf i}, C^{(n)}_{\bf i}) \}_{{\bf i} \in U, {\bf i} \neq \emptyset}$ distributed according to $\nu_n$, and $(Q^{(n)}_\emptyset, N^{(n)}_\emptyset)$ is conditionally independent of this sequence and is distributed according to $\nu_n^*$. Note that unconditionally, the $\nu_n$ and $\nu_n^*$ are random elements of $\mathscr{P}_1(\mathbb{R}^3)$ and $\mathscr{P}_1(\mathbb{R}^2)$, respectively (e.g., the empirical measures constructed from the degree sequences).

In the statement of the theorem below, we assume that $\nu$ and $\nu^*$ are the probability measures of the vectors $(Q, N, C)$ and $(Q_\emptyset, N_\emptyset)$, respectively, with $C$ independent of $(Q, N)$, i.e., the limiting weighted branching tree is a delayed weighted branching process. The measures $\nu$ and $\nu^*$ are fixed elements of $\mathscr{P}_1(\mathbb{R}^3)$ and $\mathscr{P}_1(\mathbb{R}^2)$, respectively, and hence are independent of $\mathscr{F}_n$. The symbol $ \stackrel{P}{\to}$ denotes convergence in probability.

\begin{theo} \label{T.MainWBTExt}
Conditionally on $\mathscr{F}_n$, define the processes $\{R^{(n,k)}; k \geq 0\}$, $n \geq 1$, according to \eqref{eq:Rprocess}. Similarly, define $\{R^{(k)}: k \geq 0\}$. Suppose that as $n \to \infty$,
$$d_1(\nu_n^*, \nu^*) + d_1(\mu_n, \mu) \stackrel{P}{\longrightarrow} 0.$$
Then for any fixed $k\in\mathbb N$
$$R^{(n,k)} \Rightarrow R^{(k)}, \qquad n \to \infty. $$
Moreover, if $\rho = E[N]E[|C|] < 1$, then for any $k_n \in \mathbb{N}$ such that $k_n \to \infty$ as $n \to \infty$, we have
$$R^{(n, k_n)} \Rightarrow \mathcal{R}, \qquad n \to \infty,$$
where $\mathcal{R} = \sum_{k=0}^\infty \sum_{{\bf i} \in A_k} \Pi_{\bf i} Q_{\bf i}$ is the a.s. limit of $R^{(k)}$ as $k \to \infty$.
\end{theo}

\begin{remarks}
(i) Because we allow $\nu^*$ to be different than $\nu$, the random variable $\mathcal{R}$ appearing in the limit can be written as
$$\mathcal R = \sum_{i=1}^{N_{\emptyset}} C_iR_i+Q_{\emptyset},$$
where the $\{R_i\}$ are i.i.d.~copies of the special endogenous solution $R$ to the linear SFPE \eqref{eq:SFPE2}, independent of $(Q_\emptyset, N_\emptyset, C_1, C_2, \dots)$, and with the $\{C_i\}$ i.i.d.~and independent of $(Q_\emptyset, N_\emptyset)$. In other words, $\mathcal{R}$ is a linear combination of i.i.d.~copies of the special endogenous solution to \eqref{eq:SFPE2}.

(ii) No restrictions need to be imposed on $\nu^*$, besides $\int \| {\bf x} \|_1 d\nu^*({\bf x}) < \infty$, since $C_\emptyset$ does not appear in the definitions of $W^{(j)}$ and $R^{(k)}$.

(iii) An important observation is that we have replaced the convergence in the Kantorovich-Rubinstein distance in Theorem \ref{T.MainNonHomo} with weak convergence, this is due to the fact that the proof of Theorem~\ref{T.MainWBTExt} requires that we apply Proposition \ref{P.CouplingWBT} conditionally on $\mathscr{F}_n$, which if done directly starting with $E[ | W^{(n,j)}- W^{(j)} | ]$ would lead to having to compute all the moments of $\rho_n = E[ \left. N^{(n)} |C^{(n)}| \right| \mathscr{F}_n ]$, which in general may not even be finite. Therefore, Proposition \ref{P.CouplingWBT} needs to be used after having guaranteed that $\rho_n$ is sufficiently close to $\rho$, hence the weaker mode of convergence.

(iv) The previous remark also implies that if $\rho_n \leq c < 1$ a.s., then the weak convergence in Theorem~\ref{T.MainWBTExt} can be replaced by convergence in the Kantorovich-Rubinstein sense.
\end{remarks}

\subsection{Analyzing the configuration model}

Our second example is also related to the analysis of random graphs. As mentioned earlier, when analyzing the properties of random graphs, e.g., connectivity, existence of a giant component, typical distances, phase transitions, etc., one of the basic techniques consists in coupling a ``graph exploration process" with a branching process. This is true for the configuration model (directed or undirected) as well as for other random graph models such as the Erd\H{o}s-R\'enyi graph or the generalized random graph. In line with our previous example for the analysis of ranking algorithms, we show an application of Theorem \ref{T.MainHomo} that can be used for analyzing the properties of the configuration model.

Consider an undirected graph with $n$ nodes generated according to the configuration model. In \cite{Hof_Hoo_Van_05} the authors provide an asymptotic characterization, as $n \to \infty$, of the hop count (length of the shortest path) between two randomly chosen nodes in the graph. In particular, they show that conditional on the two nodes belonging to the same component, this distance grows logarithmically in $n$. The main step of the proof consists in coupling a breadth-first exploration process, where starting from a randomly chosen node we sequentially uncover all nodes at distance one, then those at distance two, etc., with a delayed branching process (Galton-Watson process). This delayed branching process, as in our previous example, is constructed conditionally on the sigma-algebra $\mathscr{F}_n$ generated by the degree sequence.

If we denote by $\{ Z^{(n,j)}: j \geq 0\}$ be number of individuals in the $j$th generation of the coupled branching process (obtained by setting $Q_{\bf i}^{(n)} \equiv C_{\bf i}^{(n)} \equiv 1$ for all ${\bf i} \in U$ in \eqref{eq:Wprocess}), then the goal is to show that $Z^{(n,j)} /(E[N^{(n)}| \mathscr{F}_n])^j$ converges to a limit as $n,j \to \infty$. The following version of Theorem~\ref{T.MainHomo} provides such limit; here $\nu_n$ denotes the random probability measure of the conditionally i.i.d.~random variables $\{N^{(n)}_{\bf i}\}_{{\bf i} \in U, {\bf i} \neq \emptyset}$ and $\nu_n^*$ denotes that of $N_\emptyset^{(n)}$.

\begin{theo} \label{T.MainWBThomo}
Suppose there exist probability measures $\nu$ and $\nu^*$ such that
$$d_1(\nu_n, \nu) \stackrel{P}{\longrightarrow} 0 \qquad \text{and} \qquad d_1(\nu_n^*, \nu^*) \stackrel{P}{\longrightarrow} 0$$
as $n \to \infty$. Let $\{ Z^{(n,j)}: j \geq 0\}$, $n \geq 1$, be the (delayed) branching process defined by the sequence $\{N^{(n)}_{\bf i} \}_{{\bf i} \in U}$, and let $\{Z^{(j)}: j \geq 0\}$ be the one defined using $\{ N_{\bf i}\}_{{\bf i} \in U}$. Then, there exists a coupling of $\{Z^{(n,j)}: j \geq 0\}$ and $\{Z^{(j)}: j \geq 0\}$ such that for any $j_n \in \mathbb{N}$ satisfying
$$j_n \, d_1(\nu_n, \nu) \stackrel{P}{\longrightarrow} 0,$$
we have that
$$\max_{1 \leq j \leq j_n} \left| \frac{Z^{(n,j)}}{m_n^* m_n^{j-1}} - \frac{Z^{(j)}}{m^* m^{j-1}} \right| \stackrel{P}{\longrightarrow} 0 \qquad \text{and} \qquad \max_{1 \leq j \leq j_n} \frac{\left| Z^{(n,j)} - Z^{(j)} \right|}{m^{j-1}} \stackrel{P}{\longrightarrow} 0,$$
as $n \to \infty$, where $m_n = E[N^{(n)} | \mathscr{F}_n]$, $m_n^* = E[N^{(n)}_\emptyset | \mathscr{F}_n]$, $m = E[N]$, and $m^* = E[ N_\emptyset]$.
\end{theo}

In particular, Theorem \ref{T.MainWBThomo} can be used to obtain that
$$\frac{Z^{(n,j_n)}}{m_n^* m_n^{j_n-1}} \Rightarrow \mathcal{W}, \qquad n \to \infty,$$
where $\mathcal{W}$ is the a.s. limit of the martingale $Z^{(j)}/(m^* m^{j-1})$ as $j \to \infty$.

Our last result in the paper shows how large we can take $j_n$ in Theorem \ref{T.MainWBThomo} when analyzing the configuration model using a degree sequence $\{D_1, D_2, \dots, D_n\}$ of i.i.d.~random variables having common probability mass function $f$, a model used in \cite{Hof_Hoo_Van_05} to analyze the typical distance between nodes in the graph. In this context, $\mathscr{F}_n = \sigma(D_1, D_2, \dots, D_n)$,
$$\nu_n^*(\{k\}) = P(N_\emptyset^{(n)} = k | \mathscr{F}_n) = \frac{1}{n} \sum_{i=1}^n 1(D_i = k), \qquad k = 0, 1,2, \dots,$$
and
$$\nu_n(\{k\}) = P(N_1^{(n)} = k | \mathscr{F}_n) =  \sum_{i=1}^n \frac{D_i}{L_n} 1(D_i = k+1), \qquad k = 0, 1, 2, \dots,$$
with $L_n = \sum_{j=1}^n D_j$. The measure $\nu_n$ corresponds to the so-called {\em size-biased} empirical distribution. The limiting probability measures are given by
$$\nu^*(\{k\}) = f(k)  \qquad \text{and} \qquad \nu(\{k\}) = \frac{E[D 1(D=k+1)]}{E[D]}, \qquad k = 0, 1, 2, \dots,$$
where $D$ is distributed according to $f$.

\begin{lemma}\label{L.SizeBiased}
Suppose that $E[ D^{2+\epsilon} ] < \infty$ for some $\epsilon > 0$, then
$$n^{\delta'}\, d_1(\nu_n^*, \nu^*) \stackrel{P}{\longrightarrow} 0 \qquad \text{and} \qquad n^\delta \, d_1(\nu_n, \nu) \stackrel{P}{\longrightarrow} 0 \qquad n \to \infty,$$
for any $0 < \delta' < 1/2$ and  $0 < \delta < \min\{ 1/2, \epsilon/ (2+\epsilon) \}$.
\end{lemma}

\section{Proofs} \label{S.Proofs}

This last section contains the proofs of all the results presented throughout the paper. For the reader's convenience they are organized according to the section in which their statements appear.

\subsection{Bounds for the Kantorovich-Rubnistein distance}

We first prove Proposition \ref{P.CouplingWBP}, which bounds the Kantorovich-Rubinstein distance of the linear processes on two coupled weighted branching processes, by the same distance of their generic branching vectors.

\begin{proof}[Proof of Proposition \ref{P.CouplingWBP}]
Define $\mathcal{E} = E_\pi [ | \hat Q - Q | + \sum_{i=1}^\infty | \hat B_i - B_i | ]$, where the vector $(Q, B_1, B_2, \dots, \hat Q, \hat B_1, \hat B_2, \dots)$ is distributed according to $\pi$. Recall that the weights $\Pi_{\bf i}$ and $\hat \Pi_{\bf i}$ follow the recursions
$$\Pi_{({\bf i}, j)} = \Pi_{\bf i} B_{({\bf i}, j)} \qquad \text{and} \qquad \hat \Pi_{({\bf i}, j)} = \hat \Pi_{\bf i} \hat B_{({\bf i}, j)},$$
with $\Pi_{\emptyset} = \hat \Pi_{\emptyset} = 1$. Now note that for $j = 0$ we have
$$E \left[ \left| \hat W^{(0)} - W^{(0)} \right| \right] = {E} \left[ \left| \hat Q- Q\right| \right] \leq \mathcal{E}.$$
To analyze the expression for $j \geq 1$, define for $r \geq 1$, $W^{(j-1)}_r= \sum_{(r,\mathbf{i}) \in \mathbb{N}_+^j}Q_{(r,{\bf i})} \Pi_{(r,{\bf i})}/ B_r $ and \linebreak $\hat W^{(j-1)}_r=\sum_{(r,\mathbf{i}) \in \mathbb{N}_+^j} \hat Q_{(j,{\bf i})} \hat \Pi_{(r,{\bf i})}/ \hat B_r $.
We then have
$$\hat W^{(j)} = \sum_{r=1 }^{\infty} \hat B_r \hat W^{(j-1)}_r \qquad \text{and} \qquad W^{(j)} = \sum_{r=1 }^{\infty}B_r W^{(j-1)}_{r}.$$
Next, note that
\begin{align*}
    {E} \left[ \left| \hat W^{(j)} - W^{(j)} \right| \right] &\leq \sum_{r=1}^{\infty}   E \left[ \left| \hat B_r \hat W^{(j-1)}_r - B_r W^{(j-1)}_{r}\right| \right]  \\
    &\leq \sum_{r = 1}^\infty  \left\{  E \left[ \left| W^{(j-1)}_{r}(\hat B_r-B_r)\right|+\left|\hat B_r\left(\hat W^{(j-1)}_{r} - W^{(j-1)}_{r}\right)\right| \right] \right\} \\
    &\leq \sum_{r = 1}^\infty E\left[ \left|\hat B_r-B_r\right| \right] E \left[ \left| W^{(j)}_{r}\right|\right]+ \sum_{r = 1}^\infty E\left[ |\hat B_r| \right]E\left[\left|\hat W^{(j)}_{r} - W^{(j)}_{r}\right|\right] \\
    &\leq   E\left[ \left| W^{(j-1)} \right|\right] \mathcal{E} + \hat\rho E\left[\left|\hat W^{(j-1)} - W^{(j-1)} \right|\right] ,
\end{align*}
where we used the independence of the root vectors and their offspring, the observation that the random variables $\{W^{(j-1)}_r\}_{r \geq 1}$ are i.i.d.~with the same distribution as $W^{(j-1)}$ and $\{ (\hat W^{(j-1)}_{r} - W^{(j-1)}_{r})\}_{r \geq 1}$ are i.i.d.~with the same distribution as $\hat W^{(j-1)} - W^{(j-1)}$. Moreover,
\begin{equation*}
    E\left[\left| W^{(j-1)}\right| \right]\le E\left[\left|Q\right|\right]\sum_{\mathbf{i}\in \mathbb{N}_+^{j-1}}E\left[ \left|\Pi_{\mathbf{i}}\right| \right]=E\left[\left|Q\right|\right] \rho^{j-1}.
\end{equation*}
It follows that
\begin{align*}
    {E} \left[ \left| \hat W^{(j)} - W^{(j)} \right| \right]&\le E\left[\left|Q\right|\right]\rho^{j-1} \mathcal{E} +\hat\rho {E} \left[ \left| \hat W^{(j-1)} - W^{(j-1)} \right| \right] \quad \text{after $(j-1)$ iterations}\\
    &\le \left(\hat\rho^j +E\left[ \left|Q\right| \right]\sum_{t=0}^{j-1} \rho^t\hat \rho^{j-1-t}\right) \mathcal{E}.
\end{align*}
This completes the proof.
\end{proof}

Similarly we can prove an upper bound for weighted branching trees.

\begin{proof}[Proof of Proposition \ref{P.CouplingWBT}]
We construct the processes $\hat W^{(j)}$ and $W^{(j)}$ on two weighted branching trees using a coupled vector $(Q_\emptyset, N_\emptyset, \hat Q_\emptyset, \hat N_\emptyset)$ for the root nodes $\emptyset$, distributed according to $\pi^*$, and a sequence of i.i.d.~random vectors $\{(Q_{\bf i}, N_{\bf i}, C_{\bf i}, \hat Q_{\bf i}, \hat N_{\bf i}, \hat C_{\bf i}) )\}_{{\bf i} \in U, {\bf i} \neq \emptyset}$, independent of $(Q_\emptyset, N_\emptyset, \hat Q_{\emptyset}, \hat N_\emptyset)$, distributed according to $\pi$ for all other nodes.

Next, for ${\bf i} \in \mathbb{N}^k_+$, $k \geq 1$, let $B^{(0)}_{\bf i} = C_{\bf i} Q_{\bf i}$, $\hat B^{(0)}_{\bf i} = \hat C_{\bf i} \hat Q_{\bf i}$, $B^{(j)}_{\bf i} = C_{\bf i} 1(N_{\bf i} \geq i)$, and $\hat B^{(j)}_{\bf i} = \hat C_{\bf i} 1(\hat N_{\bf i} \geq i)$, for $j \geq 1$, and note that
\begin{align*}
\Pi_{\bf i} Q_{\bf i} &= Q_{\bf i} \prod_{r = 1}^k C_{{\bf i}|r} 1( i_r \leq N_{{\bf i}|r-1} )
= 1(i_1 \leq N_\emptyset) \prod_{r = 1}^{k-1} B^{(i_{r+1})}_{{\bf i}|r} B^{(0)}_{\bf i},
\end{align*}
and similarly,
$$\hat \Pi_{\bf i} \hat Q_{\bf i} = 1(i_1 \leq \hat N_\emptyset) \prod_{r = 1}^{k-1} \hat B^{(i_{r+1})}_{{\bf i}|r} \hat B^{(0)}_{\bf i},$$
with the convention that $\prod_{i=a}^b x_i \equiv 1$ if $b < a$.

Let $\mathcal{E}^* = E_{\pi^*} [ | \hat Q - Q| + |\hat N - N| ]$, where $(Q, N, \hat Q, \hat N)$ is distributed according to $\pi^*$, and $\mathcal{E} = E_\pi [ \sum_{i=0}^\infty | \hat B^{(i)} - B^{(i)} | ]$, where $(B^{(0)}, B^{(1)}, B^{(2)}, \dots, \hat B^{(0)}, \hat B^{(1)}, \hat B^{(2)}, \dots)$ is distributed according to $\pi$. It follows that
$$E \left[ \left| \hat W^{(0)} - W^{(0)} \right| \right] = E_{\pi^*} \left[ \left| \hat Q - Q \right| \right] \leq \mathcal{E}^*,$$
and for $j \geq 1$,
\begin{align*}
E\left[ \left| \hat W^{(j)} - W^{(j)} \right| \right] &= E\left[ \left|  \sum_{{\bf i} \in \mathbb{N}_+^j} 1(i_1 \leq \hat N_\emptyset) \prod_{r = 1}^{j-1} \hat B^{(i_{r+1})}_{{\bf i}|r} \hat B^{(0)}_{\bf i} - \sum_{{\bf i} \in \mathbb{N}_+^j} 1(i_1 \leq N_\emptyset) \prod_{r = 1}^{j-1} B^{(i_{r+1})}_{{\bf i}|r} B^{(0)}_{\bf i}   \right| \right] \\
&\leq \sum_{{\bf i} \in \mathbb{N}_+^j} E\left[ \left| 1(i_1 \leq \hat N_{\emptyset}) \prod_{r = 1}^{j-1} \hat B^{(i_{r+1})}_{{\bf i}|r} \left( \hat B^{(0)}_{\bf i} -  B^{(0)}_{\bf i} \right)  \right| \right] \\
&\hspace{5mm} + \sum_{{\bf i} \in \mathbb{N}_+^j} E\left[ \left| \left( 1(i_1 \leq \hat N_{\emptyset}) - 1(i_1 \leq  N_\emptyset)  \right)  \prod_{r = 1}^{j-1} \hat B^{(i_{r+1})}_{{\bf i}|r}  B^{(0)}_{\bf i}   \right| \right] \\
&\hspace{5mm} + \sum_{{\bf i} \in \mathbb{N}_+^j} E\left[ \left| 1(i_1 \leq N_{\emptyset}) \left(  \prod_{r = 1}^{j-1} \hat B^{(i_{r+1})}_{{\bf i}|r}  - \prod_{r = 1}^{j-1} B^{(i_{r+1})}_{{\bf i}|r} \right) B^{(0)}_{\bf i}   \right| \right] \\
&= \sum_{{\bf i} \in \mathbb{N}_+^j}  P(\hat N \geq i_1) \prod_{r=1}^{j-1} E\left[ |\hat C| 1 (\hat N \geq i_{r+1})  \right]  E_\pi \left[ |\hat B^{(0)} - B^{(0)}| \right] \\
&\hspace{5mm} + \sum_{{\bf i} \in \mathbb{N}_+^j} E_{\pi^*} \left[\left| 1(i_1 \leq \hat N) - 1(i_1 \leq N) \right| \right] \prod_{r=1}^{j-1} E\left[ |\hat C| 1 (\hat N \geq i_{r+1})  \right] E\left[ |CQ | \right] \\
&\hspace{5mm} +  \sum_{{\bf i} \in \mathbb{N}_+^j} P(N \geq i_1) E\left[ \left|  \prod_{r = 1}^{j-1} \hat B^{(i_{r+1})}_{{\bf i}|r}  - \prod_{r = 1}^{j-1} B^{(i_{r+1})}_{{\bf i}|r}    \right| \right] E[|CQ|],
\end{align*}
where we have used the independence among the generic branching vectors of the weighted branching trees. Moreover,
$$\sum_{{\bf i} \in \mathbb{N}_+^j}  P(\hat N \geq i_1) \prod_{r=1}^{j-1} E\left[ |\hat C| 1 (\hat N \geq i_{r+1})  \right] = \sum_{i=1}^\infty P(\hat N \geq i) \left( \sum_{k=1}^\infty E\left[ |\hat C| 1(\hat N \geq k) \right] \right)^{j-1}  = E[\hat N] \hat \rho^{j-1},$$
where $\hat \rho = E[\hat N |\hat C|]$. Similarly,
$$\sum_{i =1}^\infty E_{\pi^*} \left[\left| 1(i \leq \hat N) - 1(i \leq N) \right| \right]  = \sum_{i=1}^\infty E_{\pi^*} \left[ 1(N < i \leq \hat N) + 1(\hat N < i \leq N) \right] = E_{\pi^*} \left[ |\hat N - N| \right].$$

It follows that
\begin{align*}
E\left[ \left| \hat W^{(j)} - W^{(j)} \right| \right] &\leq E[\hat N] \hat \rho^{j-1} E_\pi \left[ |\hat B^{(0)} - B^{(0)}| \right]  + E[|CQ|]   \hat \rho^{j-1}  E_{\pi^*} \left[ |\hat N - N| \right]  \\
&\hspace{5mm} + E[N] E[|CQ|] \sum_{(i_2, i_3,\dots, i_j) \in \mathbb{N}_+^{j-1}} E\left[ \left|  \prod_{r = 1}^{j-1} \hat B^{(i_{r+1})}_{(1,i_2, \dots, i_r)}  - \prod_{r = 1}^{j-1} B^{(i_{r+1})}_{(1,i_2, \dots, i_r)}    \right| \right].
\end{align*}
To analyze the last expectation let $a_j = \sum_{(i_2, i_3,\dots, i_j) \in \mathbb{N}_+^{j-1}} E\left[ \left|  \prod_{r = 1}^{j-1} \hat B^{(i_{r+1})}_{(1,i_2, \dots, i_r)}  - \prod_{r = 1}^{j-1} B^{(i_{r+1})}_{(1,i_2, \dots, i_r)}    \right| \right]$ for $j \geq 2$, and $a_1 = 0$. It follows that for $j \geq 2$,
\begin{align*}
a_j &\leq \sum_{(i_2, i_3,\dots, i_j) \in \mathbb{N}_+^{j-1}} E\left[ \left|  \prod_{r = 1}^{j-2} \hat B^{(i_{r+1})}_{(1,i_2, \dots, i_r)}  - \prod_{r = 1}^{j-2} B^{(i_{r+1})}_{(1,i_2, \dots, i_r)}    \right| \left|  \hat B_{(1,i_2, \dots, i_{j-1})}^{(i_j)} \right|  \right] \\
&\hspace{5mm} + \sum_{(i_2, i_3,\dots, i_j) \in \mathbb{N}_+^{j-1}} E\left[ \left| \prod_{r = 1}^{j-2} B^{(i_{r+1})}_{(1,i_2, \dots, i_r)} \right| \left|  \hat B_{(1,i_2, \dots, i_{j-1})}^{(i_j)}  -  B^{(i_{j})}_{(1,i_2, \dots, i_{j-1})}    \right| \right] \\
&= a_{j-1} \sum_{i_j = 1}^\infty E\left[ |\hat C| 1(\hat N \geq i_j)  \right] + \sum_{(i_2, i_3,\dots, i_j) \in \mathbb{N}_+^{j-1}} \prod_{r=1}^{j-2} E\left[ |C| 1(N \geq i_{r+1}) \right] E_\pi \left[ \left| \hat B^{(i_j)} - B^{(i_j)} \right| \right] \\
&= \hat \rho \, a_{j-1} + \rho^{j-2} \sum_{i=1}^\infty E_\pi \left[ \left| \hat B^{(i)} - B^{(i)} \right| \right],
\end{align*}
where $\rho = E[N |C|]$. Iterating this recursion $j-2$ times gives
$$a_j \leq \hat \rho^{j-1} a_1 + \sum_{t = 0}^{j-2} \hat \rho^t \rho^{j-2-t} \sum_{i=1}^\infty E_\pi \left[ \left| \hat B^{(i)} - B^{(i)} \right| \right] = \sum_{t = 0}^{j-2} \hat \rho^t \rho^{j-2-t} \sum_{i=1}^\infty E_\pi \left[ \left| \hat B^{(i)} - B^{(i)} \right| \right] .$$

We conclude that for $j \geq 1$,
\begin{align*}
E\left[ \left| \hat W^{(j)} - W^{(j)} \right| \right] &\leq E[\hat N] \hat \rho^{j-1} E_\pi \left[ |\hat B^{(0)} - B^{(0)}| \right]  + E[|CQ|]   \hat \rho^{j-1}  E_{\pi^*} \left[ |\hat N - N| \right]  \\
&\hspace{5mm} + 1(j \geq 2) E[N] E[|CQ|] \sum_{t = 0}^{j-2} \hat \rho^t \rho^{j-2-t} \sum_{i=1}^\infty E_\pi \left[ \left| \hat B^{(i)} - B^{(i)} \right| \right] \\
&\leq \left( E[\hat N] \vee \frac{E[N] E[|CQ|]}{\rho} \right) \left( \sum_{t=0}^{j-1} \hat \rho^t \rho^{j-1-t} \right) \mathcal{E} + E[|CQ|]   \hat \rho^{j-1} \mathcal{E}^*.
\end{align*}
\end{proof}

\subsection{Convergence to the special endogenous solution}

We now proceed to prove the two main theorems of the paper, Theorems \ref{T.MainHomo} and \ref{T.MainNonHomo}.

\begin{proof}[Proof of Theorem \ref{T.MainHomo}]

{\it Case 1:} Weighted branching processes.

    Choose a coupling $\pi$ of $\mu_n$ and $\mu$ such that $E_\pi [ |Q-Q^{(n)}|+\sum_{j=1}^{\infty}|B_j-B^{(n)}_j| ] = d_1(\mu,\mu_n)$.
    If we construct both WBPs based on this optimal coupling, then by Proposition~\ref{P.CouplingWBP},
    \begin{align*}
        E \left[ \left| W^{(n,j)} - W^{(j)} \right| \right] &\le  \left(\rho_n^{j} + E\left[ \left|Q\right| \right]\sum_{t=0}^{j-1}\rho^t\rho_n^{j-1-t}\right)d_1(\mu,\mu_n) \\
        &\leq (E[|Q| \vee \rho) (j+1) \rho^{j-1} \left( 1 \vee \frac{\rho_n}{\rho} \right)^j d_1(\mu,\mu_n).
    \end{align*}

For fixed $j \geq 1$ note that $|\rho_n - \rho| \leq d_1(\mu, \mu_n)$, and hence $\left( 1 \vee (\rho_n/\rho) \right)^j \to 1$ as $n \to \infty$, which in turn implies that $E [ | W^{(n,j)} - W^{(j)} | ]\to 0$.

Assume now $Q^{(n)} = Q \equiv 1$, and $\{C_j^{(n)}, C_j\}$ are nonnegative for all $n,j$; suppose $j_n \to \infty$ and $j_n d_1(\mu, \mu_n) \to 0$ as $n \to \infty$. First note that $\{W^{(j)}/\rho^{j}\}$ is a mean one nonnegative martingale with respect to the filtration generated by $\mathcal{G}_j = \sigma( (B_{({\bf i},1)}, B_{({\bf i},2)}, \dots): {\bf i} \in \mathbb{N}_+^r, 0 \leq r < j)$, $\mathcal{G}_0 = \sigma(\varnothing)$. Therefore,
\begin{align*}
E\left[ \left| \frac{W^{(n, j_n)} }{\rho_n^{j_n}} - \frac{W^{(j_n)}}{\rho^{j_n}} \right| \right] &\leq  E\left[ \frac{1}{\rho_n^{j_n}}\left|W^{(n,j_n)} - W^{(j_n)} \right|  \right]+E\left[ \frac{W^{(j_n)}}{\rho^{j_n}} \left| \left(\frac{\rho}{\rho_n}\right)^{j_n}-1 \right|  \right] \\
&\leq \frac{(1 \vee \rho)}{\rho} (j_n+1) \left(\frac{\rho}{\rho_n} \right)^{j_n} \left( 1 \vee \frac{\rho_n}{\rho} \right)^{j_n} d_1(\mu,\mu_n) + \left| \left(\frac{\rho}{\rho_n}\right)^{j_n}-1 \right| \\
&\leq  \frac{(1\vee \rho)}{\rho} (j_n+1)  e^{j_n(\rho/\rho_n-1)^+} d_1(\mu,\mu_n) + j_n \left| \frac{\rho}{\rho_n} - 1 \right| e^{(j_n-1) (\rho/\rho_n - 1)^+},
\end{align*}
where in the last step we used the inequalities
\begin{equation} \label{eq:PowerBounds}
(x \vee 1)^j \leq e^{j(x-1)^+} \qquad \text{ and } \qquad |x^j -1| \leq j|x-1| e^{(j-1)(x-1)^+} \qquad \text{ for all } x > 0, j \in \mathbb{N}.
\end{equation}
 Since $j_n d_1(\mu,\mu_n) \to 0$ as $n \to \infty$, then so does $j_n |\rho/\rho_n - 1| \to 0$ as $n \to \infty$, and we conclude that the expected value converges to zero. Since by the martingale convergence theorem $W^{(j_n)}/\rho^{j_n} \to \mathcal{W}$ almost surely, then
 $$\frac{W^{(n, j_n)} }{\rho_n^{j_n}} \Rightarrow \mathcal{W}, \qquad n\to \infty.$$
If $E[\mathcal{W}] = 1$ then $E[ |W^{(j_n)}/\rho^{j_n} - \mathcal{W} | ] \to 0$ and we can replace the convergence in distribution to convergence in the Kantorovich-Rubinstein distance.

The last statement of the theorem for weighted branching processes follows from noting that
\begin{align*}
\frac{1}{\rho^j}  E\left[ \left| W^{(n,j)} - W^{(j)} \right| \right] &\leq   \left| \frac{1}{\rho^j} - \frac{1}{\rho_n^j} \right| E\left[ |W^{(n,j)}| \right] + E\left[ \left| \frac{W^{(n, j_n)} }{\rho_n^{j_n}} - \frac{W^{(j_n)}}{\rho^{j_n}} \right| \right] \\
&= \left| \left(\frac{\rho}{\rho_n}\right)^{j}-1 \right| + E\left[ \left| \frac{W^{(n, j_n)} }{\rho_n^{j_n}} - \frac{W^{(j_n)}}{\rho^{j_n}} \right| \right],
\end{align*}
which were already shown to converge to zero for all $0 \leq j \leq j_n$. This completes the case.

{\it Case 2:} Weighted branching trees.

    Construct versions of the processes $\{W^{(n,j)} : j \geq 0 \}$ and $\{ W^{(j)} : j \geq 0\}$ using a sequence of coupled vectors $\{ (Q^{(n)}_{\bf i}, N^{(n)}_{\bf i}, C^{(n)}_{\bf i}, Q_{\bf i}, N_{\bf i}, C_{\bf i})\}_{{\bf i} \in U, {\bf i} \neq \emptyset}$ according to the coupling $\pi$ satisfying $d_1(\mu_n,\mu) = E_\pi[ |C^{(n)} Q^{(n)} - CQ| + \sum_{i=1}^\infty | C^{(n)} 1(N^{(n)} \geq i) - C 1(N \geq i)| ]$. Let the root vector $(Q^{(n)}_\emptyset, N^{(n)}_\emptyset, Q_\emptyset, N_\emptyset)$ be distributed according to $\pi^*$, where $d_1(\nu_n^*, \nu^*) = E_{\pi^*} [ |Q^{(n)} - Q| + |N^{(n)} - N| ]$, and be independent of all other nodes.

By Proposition \ref{P.CouplingWBT} we have $E[ | W^{(n,0)} - W^{(0)} | ] \leq d_1(\nu_n^*, \nu^*)$ and
$$E\left[ \left| W^{(n,j)} -  W^{(j)} \right| \right] \leq K j (\rho_n \vee \rho)^{j-1} d_1(\mu_n, \mu) + K \rho_n^{j-1} d_1(\nu_n^*, \nu^*), \qquad j \geq 1,$$
with $K = \max\{ E[N^{(n)}], E[|Q|] \}$. Note that
$$|\rho_n - \rho|  = \left| \sum_{i=1}^\infty E\left[ C^{(n)} 1(N^{(n)} \geq i) - C 1(N \geq i) \right] \right| \leq d_1(\mu_n,\mu).$$
The result for fixed $j$ follows immediately.

Assume now that $Q^{(n)} = Q = 1$, and $\{ C^{(n)}, C\}$ are nonnegative, and recall that $C$ is independent of $(Q, N)$, and therefore $\mu$ defines a weighted branching process. This in turn implies that $\{ W^{(j)}/\rho^j\}$ is a nonnegative martingale with respect to the filtration generated by $\mathcal{H}_j = \sigma( (N_{\bf i}, C_{({\bf i},1)}, \dots, C_{({\bf i}, N_{\bf i})}): {\bf i} \in A_r, 0 \leq r < j)$, $\mathcal{H}_0 = \sigma(\varnothing)$. It follows that
    \begin{align*}
E\left[\left| \frac{W^{(n,j_n)}}{\rho_n^{j_n}} - \frac{W^{(j_n)}}{\rho^{j_n}} \right| \right] &\leq E\left[  \left| W^{(n,j_n)} - W^{(j_n)} \right| \right]  \frac{1}{\rho_n^{j_n}} + \left| \left( \frac{\rho}{\rho_n} \right)^{j_n} - 1  \right|  \\
&\leq K j_n  \left( \frac{\rho \vee \rho_n}{\rho_n} \right)^{j_n} d_1(\mu_n, \mu) + \frac{K}{\rho_n} d_1(\nu_n^*, \nu^*) + \left| \left( \frac{\rho}{\rho_n} \right)^{j_n} - 1  \right| \\
&\leq K  j_n e^{j_n (\rho/\rho_n - 1)^+} d_1(\mu_n, \mu) + \frac{K}{\rho_n} d_1(\nu_n^*, \nu^*) + j_n \left| \frac{\rho}{\rho_n} - 1 \right| e^{(j_n-1)(\rho/\rho_n-1)^+},
\end{align*}
where in the last step we used the inequalities \eqref{eq:PowerBounds}. This last expression converges to zero since $j_n d_1(\mu_n, \mu) \to 0$ as $n \to \infty$.

The proof of the last statement is identical to that of {\it Case 1} and is therefore omitted.
\end{proof}

\bigskip

We now proceed to the nonhomogeneous case.

\begin{proof}[Proof of Theorem~\ref{T.MainNonHomo}]

{\it Case 1:} Weighted branching processes.

The result for fixed $k$ follows from Theorem \ref{T.MainNonHomo}, since
$$\left| R^{(n,k)} - R^{(k)} \right| =  \left| \sum_{j=0}^k \left( W^{(n,j)} - W^{(j)} \right) \right| \leq \sum_{j=0}^k \left| W^{(n,j)} - W^{(j)} \right|.$$
If in addition we have $\rho<1$, then, by Proposition \ref{P.CouplingWBP} (using the optimal coupling),
\begin{align*}
    E\left[\left| R^{(n,k_n)}-R^{(k_n)}\right| \right]&\le \sum_{j=0}^{k_n}E\left[ \left|W^{(n,j)}-W^{(j)}\right| \right] \le \sum_{j=0}^{k_n}\left(\rho_n^{j} + E\left[ \left|Q\right| \right]\sum_{t=0}^{j-1}\rho^t\rho_n^{j-1-t}\right)d_1(\mu_n,\mu).
\end{align*}
Now note that since $|\rho_n - \rho| \leq d_1(\mu_n, \mu)$, then for any $0 < \varepsilon < 1-\rho$ we have that $\rho_n < 1 - \varepsilon$ for all $n$ sufficiently large. In this case,
\[
    E\left[\left| R^{(n,k_n)}-R^{(k_n)}\right| \right]
    \le \sum_{j=0}^{\infty}\left((1-\varepsilon)^{j} + E\left[ \left|Q\right| \right] j (1-\varepsilon)^{j-1} \right)d_1(\mu_n,\mu)\to 0,
\]
as $n \to \infty$ for any $k_n \geq 1$. Since we also have that
$$E\left[ \left| R^{(k_n)} - R \right| \right] = E\left[ \left| \sum_{j = k_n+1}^\infty W^{(j)} \right| \right] \leq \sum_{j=k_n+1}^\infty E[|Q|] \rho^j = \frac{E[|Q|] \rho^{k_n+1}}{1-\rho},$$
then for any $k_n \to \infty$,
$$R^{(n,k_n)} \stackrel{d_1}{\longrightarrow} R, \qquad n \to \infty.$$

{\it Case 2:} Weighted branching trees.

The proof of the result for fixed $k$ follows from Theorem \ref{T.MainHomo} as before. For $k_n$ and $\rho < 1$ we use Proposition \ref{P.CouplingWBT} (using the optimal couplings $\pi^*$ and $\pi$) to obtain
\begin{align*}
E\left[\left| R^{(n,k_n)}-R^{(k_n)}\right| \right] &\le \sum_{j=0}^{k_n}E\left[ \left|W^{(n,j)}-W^{(j)}\right| \right] \\
&\leq d_1(\nu_n^*, \nu^*) + K \sum_{j=1}^{k_n} \left( \sum_{t=0}^{j-1} \rho_n^t \rho^{j-1-t} d_1(\mu_n,\mu) + \rho_n^{j-1} d_1(\nu_n^*, \nu^*) \right),
\end{align*}
with $K = \max\{ E[N^{(n)}], E[|Q|]\}$. Using the same arguments from {\it Case 1} note that for any $0 < \varepsilon < 1-\rho$ and $n$ sufficiently large,
\begin{align*}
E\left[\left| R^{(n,k_n)}-R^{(k_n)}\right| \right] &\le d_1(\nu_n^*, \nu^*) + K \sum_{j=1}^\infty \left( j (1-\varepsilon)^{j-1}   d_1(\mu_n,\mu) + (1-\varepsilon)^{j-1} d_1(\nu_n^*, \nu^*) \right) \to 0,
\end{align*}
as $n \to \infty$ for any $k_n \geq 1$. The rest of proof is the same as that of {\it Case 1} and is therefore omitted.
\end{proof}
\bigskip

The last result we need to prove in this section is Lemma \ref{L.FinalConditionsWBT}.

\begin{proof}[Proof of Lemma \ref{L.FinalConditionsWBT}]
From the definition of the Kantorovich-Rubinstein metric and the fact that the infimum is always attained (see, e.g., \cite{Villani_2009}, Theorem 4.1), there exists a coupling $\pi$ of $(N^{(n)}, Q^{(n)}, C^{(n)}, N, Q, C)$ such that
\begin{equation} \label{eq:OptimalCoupling}
d_1(\nu_n, \nu) =  E_\pi \left[ | Q^{(n)} - Q| + |N^{(n)} - N| + | C^{(n)} - C| \right].
\end{equation}
Next, define the vectors
$${\bf Y}_n = C^{(n)} (Q^{(n)}, 1(N^{(n)} \geq 1), 1(N^{(n)} \geq 2), \dots) \qquad \text{and} \qquad {\bf Y} = C (Q, 1(N \geq 1),  1(N \geq 2), \dots).$$

We will first show that $\| {\bf Y}_n - {\bf Y} \|_1 \stackrel{P}{\to} 0$ as $n \to \infty$. To this end, let $(\hat Q, \hat N, \hat C) = (Q^{(n)}, N^{(n)}, C^{(n)})$ to simplify the notation and define $X_n = \| (N^{(n)}, Q^{(n)}, C^{(n)}) - (N, Q, C) \|_1$. Note that \eqref{eq:OptimalCoupling} implies that $X_n \to 0$ in mean, and therefore in probability. Now note that
\begin{align*}
\left\| {\bf Y}_n - {\bf Y} \right\|_1 &=  |\hat Q \hat C - Q C| + \sum_{i=1}^\infty | \hat C 1(\hat N \geq i) - C 1(N \geq i)|  \\
&=  |\hat Q \hat C - Q C| + \sum_{i=1}^\infty \left( |\hat C - C| 1(i \leq \hat N \wedge N) + |\hat C| 1(N < i \leq \hat N) + |C| 1(\hat N < i \leq N) \right)  \\
&=  |\hat Q \hat C - Q C| +  |\hat C - C| (\hat N \wedge N) + |\hat C| (\hat N - N)^+ + |C| (N - \hat N)^+ \\
&\leq |\hat C| |\hat Q - Q| + |Q| |\hat C - C| + |\hat C - C| (\hat N \wedge N) + |\hat C| (\hat N - N)^+ + |C| (N - \hat N)^+ \\
&\leq \left(2 |C^{(n)}| + |Q| + N + |C| \right) X_n \stackrel{P}{\to} 0, \qquad n \to \infty,
\end{align*}
by the converging together lemma. It remains to show that $\left\| {\bf Y}_n - {\bf Y} \right\|_1 \to 0$ in mean.

By the triangle's inequality we have that
$$Q_n \triangleq \left| \left\| {\bf Y}_n \right\|_1 - \left\|  {\bf Y} \right\|_1 \right| \leq \left\| {\bf Y}_n - {\bf Y} \right\|_1 \stackrel{P}{\to} 0, \qquad n \to \infty.$$
Also, by assumption,
\begin{align*}
E\left[ \| {\bf Y}_n \|_1 \right] &= E\left[ |\hat Q\hat C| + \sum_{i=1}^\infty |\hat C| 1(\hat N \geq i) \right] =  E\left[ |\hat Q \hat C|+ | \hat C | \hat N \right] \to E[ |CQ| + |C|N] = E\left[ \| {\bf Y} \|_1 \right]
\end{align*}
as $n \to \infty$, and therefore $E[Q_n] \to 0$ (see, e.g., Theorem 5.5.2 in \cite{Durrett_2010}). Now note that since $\left\| {\bf Y}_n - {\bf Y} \right\|_1 \leq \left\| {\bf Y}_n \right\|_1 + \left\|  {\bf Y} \right\|_1 \leq Q_n + 2 \left\|  {\bf Y} \right\|_1$, we have
\begin{align*}
E \left[ \left\| {\bf Y}_n - {\bf Y} \right\|_1  \right] \leq E \left[ \left\| {\bf Y}_n - {\bf Y} \right\|_1 1( Q_n \leq 1)  \right]  + E[Q_n] + 2 E\left[  \| {\bf Y} \|_1 1( Q_n > 1) \right],
\end{align*}
where $\left\| {\bf Y}_n - {\bf Y} \right\|_1 1( Q_n \leq 1)$ and $\| {\bf Y} \|_1 1( Q_n > 1)$ are uniformly integrable by Theorem 13.3 in \cite{Williams_1991}, and hence
$$\lim_{n \to \infty} E \left[ \left\| {\bf Y}_n - {\bf Y} \right\|_1 1( Q_n \leq 1)  \right] = \lim_{n \to \infty} E\left[  \| {\bf Y} \|_1 1( Q_n > 1) \right] = 0.$$
\end{proof}

\subsection{Applications}

In this last section of the paper we prove the theorems regarding our applications to the analysis of information ranking algorithms and random graphs.

\begin{proof}[Proof of Theorem \ref{T.MainWBTExt}]
By applying the same steps from the proof of Theorem~\ref{T.MainNonHomo} conditionally on the sigma-algebra $\mathscr{F}_n$ (with $\nu_n^*$ the probability distribution of the root vector, which is allowed to be different) we obtain
\begin{align*}
E\left[ \left. \left| R^{(n,k)}-R^{(k)}\right| \right| \mathscr{F}_n \right] &\le \sum_{j=0}^{k} E\left[ \left. \left|W^{(n,j)}-W^{(j)}\right| \right| \mathscr{F}_n \right] \\
&\leq d_1(\nu_n^*, \nu^*) + K_n \sum_{j=1}^{k} \left( \sum_{t=0}^{j-1} \rho_n^t \rho^{j-1-t} d_1(\mu_n,\mu) + \rho_n^{j-1} d_1(\nu_n^*, \nu^*) \right),
\end{align*}
where $K_n = \max\{ E[N_\emptyset^{(n)} | \mathscr{F}_n], E[|Q|]\}$. Next, for any $\varepsilon > 0$ define the event $A_{n,\varepsilon} = \{ d_1(\nu_n^*, \nu^*) + d_{1}(\mu_n,\mu)\le \varepsilon \}$. Recall that $\rho_n \leq \rho + d_1(\mu_n,\mu)$ and note that $K_n \leq \max\{ E[N_\emptyset] + d_1(\nu_n^*, \nu^*), E[|Q|]\}$. The result for fixed $k$ is obtained by first conditioning on the event $A_{n,\varepsilon}$, as will be done for the case of $k_n \to \infty$ below, and therefore we omit the details.

For $k_n \to \infty$ and $\rho < 1$, choose $0 < \varepsilon \leq (1-\rho)/2$ and define $A_{n,\varepsilon}$ as above; set $K = \max\{ E[N_\emptyset] + \varepsilon, E[|CQ|]\}$. Then apply Markov's inequality conditionally to obtain that for any $\delta > 0$,
\begin{align*}
P\left( \left| R^{(n,k_n)}-R^{(k_n)}\right| > \delta \right) &\le P\left( \left| R^{(n,k_n)}-R^{(k_n)}\right| > \delta, \, A_{n,\varepsilon} \right)  +P\left( A_{n,\varepsilon}^c \right) \\
&\leq \delta^{-1} E\left[ 1(A_{n, \varepsilon} ) E\left[ \left. \left| R^{(n,k_n)}-R^{(k_n)}\right| \right| \mathscr{F}_n \right]  \right]  + P\left( A_{n, \varepsilon}^c \right) \\
&\leq \delta^{-1} E\left[ 1(A_{n,\varepsilon}) \left\{ d_1(\nu_n^*, \nu^*) +  K_n \sum_{j=1}^\infty (1-\varepsilon)^{j-1} d_1(\nu_n^*, \nu^*)  \right\} \right]  \\
&\hspace{5mm} + \delta^{-1}  E\left[ 1(A_{n,\varepsilon}) K_n \sum_{j=1}^{\infty}  j (1-\varepsilon)^{j-1} d_1(\mu_n, \mu) \right]  + P\left( A_{n, \varepsilon}^c \right) \\
&\leq \delta^{-1} (1 \vee K/\varepsilon^2) E\left[ d_1(\nu_n^*, \nu^*) \wedge \varepsilon + d_1(\mu_n, \mu) \wedge \varepsilon \right] + P(A_{n,\varepsilon}^c).
\end{align*}
The assumption that $d_1(\nu_n^*, \nu^*) + d_1(\mu_n, \mu) \stackrel{P}{\to} 0$ as $n \to \infty$ and the bounded convergence theorem show that $| R^{(n,k_n)} - R^{(k_n)} | \stackrel{P}{\to} 0$. This combined with the observation that $| R^{(k_n)} - \mathcal{R} | \to 0$ almost surely complete the proof.
\end{proof}

\bigskip

We now prove the theorem for the analysis of the configuration model.

\begin{proof}[Proof of Theorem \ref{T.MainWBThomo}]
Let $F_n(k) = P(N \leq k | \mathscr{F}_n)$, $G_n(k) = P(N_\emptyset \leq k | \mathscr{F}_n)$, $F(k) = P(N \leq k)$ and $G(k) = P(N_\emptyset \leq k)$. Let $\{U_{\bf i}\}_{{\bf i} \in U}$ be a sequence of i.i.d.~uniform (0,1) random variables, independent of $\mathscr{F}_n$. Construct the two trees using the coupled vectors $\{ (N_{\bf i}^{(n)}, N_{\bf i})\}_{{\bf i} \in U}$ where
$$(N^{(n)}_\emptyset, N_\emptyset) = (G_n^{-1}(U_\emptyset), G^{-1}(U_\emptyset)) \qquad \text{and} \qquad (N^{(n)}_{\bf i}, N_{\bf i}) = (F_n^{-1}(U_{\bf i}), F^{-1}(U_{\bf i}) ), \qquad {\bf i} \neq \emptyset,$$
where $f^{-1}(t) = \inf\{x \in \mathbb{R}: f(x) \geq t\}$. These are the optimal couplings for which $d_1(\nu_n^*, \nu^*)$ and $d_1(\nu_n, \nu)$ are achieved.

Next, by adapting Proposition \ref{P.CouplingWBT} to allow $\nu_n^*$ to be different, we obtain for $j \geq 1$,
\begin{align*}
E\left[ \left. \left| Z^{(n,j)} - Z^{(j)} \right| \right| \mathscr{F}_n \right] &\leq m_n^{j-1} d_1(\nu_n^*, \nu^*) + 1(j \geq 2) \, m^* \sum_{t=0}^{j-2} m_n^t m^{j-2-t} d_1(\nu_n,\nu) ,
\end{align*}
where $m = E[N]$, $m_n = E[N^{(n)} | \mathscr{F}_n]$, $m^* = E[N_\emptyset]$, and $m_n^* = E[ N^{(n)}_\emptyset | \mathscr{F}_n]$.

Note that $M^{(n,j)} = Z^{(n,j)}/(m_n^* m_n^{j-1}) - Z^{(j)}/(m^* m^{j-1})$, for $j \geq 1$, is a martingale with respect to the filtration $\mathcal{G}_j = \sigma( N_{\bf i}^{(n)}, N_{\bf i}: {\bf i} \in \mathbb{N}^s, 0 \leq s \leq j-1)$, conditionally on $\mathscr{F}_n$. Let $A_{n,\delta} = \{ d_1(\nu_n^*, \nu^*) + j_n d_1(\nu_n, \nu) \leq \delta \}$ and note that by assumption $P(A_{n,\delta}) \to 0$ for any $\delta > 0$. Now let $\varepsilon > 0$ and use Doob's maximal inequality to obtain
  \begin{align*}
&P\left(  \max_{1 \leq j \leq j_n} \left| M^{(n,j)} \right| > \varepsilon , \, A_{n,\delta}  \right) \leq \frac{1}{\varepsilon} E\left[ 1(A_{n,\delta}) E\left[\left. \left|M^{(n,j_n)}\right| \right| \mathscr{F}_n \right] \right] .
\end{align*}
 To bound this last expectation note that for any $j \geq 1$,
  \begin{align*}
  E\left[\left. \left|M^{(n,j)}\right| \right| \mathscr{F}_n \right] &\leq  E\left[\left. \frac{1}{m^* m^{j-1}} \left| Z^{(n,j)} - Z^{(j)} \right| + \left| \frac{m_n^* m_n^{j-1}}{m^* m^{j-1}} - 1 \right| \frac{Z^{(n,j)}}{m_n^* m_n^{j-1}}  \right| \mathscr{F}_n \right] \\
  &= \frac{1}{m^* m^{j-1}}  E\left[ \left.\left| Z^{(n,j)}-Z^{(j)}\right|\right|\mathscr F_n \right] + \left| \frac{m_n^* m_n^{j-1}}{m^* m^{j-1}} - 1 \right| \\
     &\leq \frac{1}{m^*} \left( \frac{m_n}{m} \right)^{j-1} d_1(\nu_n^*, \nu^*) + 1(j \geq 2) \frac{1}{m}  \sum_{t=0}^{j-2} \left( \frac{m_n}{m} \right)^t  d_1(\nu_n,\nu) + \left| \frac{m_n^* m_n^{j-1}}{m^* m^{j-1}} - 1 \right| \\
     &\leq \frac{1}{m \wedge m^*} \left( \frac{m \vee m_n}{m} \right)^{j-1} \left( d_1(\nu_n^*, \nu^*) + j d_1(\nu_n,\nu) \right) + \left| \frac{m_n^* m_n^{j-1}}{m^* m^{j-1}} - 1 \right| .
  \end{align*}
  Moreover, by \eqref{eq:PowerBounds} we have that for $1 \leq j \leq j_n$ and on the event $A_{n,\delta}$,
  $$\left( \frac{m \vee m_n}{m} \right)^{j-1} \leq e^{(j-1) (m_n/m -1)^+} \leq e^{(j-1) d_1(\nu_n, \nu)/m} \leq e^{\delta/m}$$
  and
  \begin{align*}
\left| \frac{m_n^* m_n^{j-1}}{m^* m^{j-1}} - 1 \right| & \leq \frac{m_n^*}{m^*} \left| \left( \frac{m_n}{m} \right)^{j-1} - 1 \right| + \left| \frac{m_n^*}{m^*} - 1 \right| \\
&\leq \frac{m_n^*}{m^*} (j-1) \left| \frac{m_n}{m} - 1 \right| e^{(j-2)^+ (m_n/m - 1)^+} + \left| \frac{m_n^*}{m^*} - 1 \right| \\
&\leq \frac{(m^* + \delta) e^{\delta/m}}{m^* m} \cdot (j-1) d_1(\nu_n, \nu)  + \frac{1}{m^*} d_1(\nu_n^*, \nu^*).
\end{align*}
It follows that on the event $A_{n,\delta}$,
\begin{align*}
E\left[\left. \left|M^{(n,j_n)}\right| \right| \mathscr{F}_n \right] &\leq \left( \frac{e^{\delta/m}}{m \wedge m^*} + \frac{1}{m^*} \right)  d_1(\nu_n^*, \nu^*) +    \left( \frac{e^{\delta/m}}{m\wedge m^*}   + \frac{(m^* + \delta) e^{\delta/m}}{m^* m} \right)  j_n d_1(\nu_n, \nu).
\end{align*}
Hence,
$$P\left(  \max_{1 \leq j \leq j_n} \left| M^{(n,j)} \right| > \varepsilon , \, A_{n,\delta}  \right) \leq \frac{K}{\varepsilon} E\left[ 1(A_{n,\delta}) \left( d_1(\nu_n^*, \nu^*) + j_n d_1(\nu_n, \nu) \right) \right],$$
for some constant $K = K(\delta) < \infty$. Since $d_1(\nu_n^*, \nu^*) + j_n d_1(\nu_n, \nu) \stackrel{P}{\to} 0$ as $n \to \infty$, the bounded convergence theorem gives
\begin{equation} \label{eq:MGconvergence}
\max_{1\leq j \leq j_n} \left| M^{(n,j)} \right| \stackrel{P}{\longrightarrow} 0, \qquad n \to \infty.
\end{equation}

For the second statement of the theorem note that for any $j \geq 1$,
\begin{align*}
 \frac{\left| Z^{(n,j)} - Z^{(j)} \right|}{m^{j-1}} &\leq  m^* \left| M^{(n,j)} \right| + m^* \left| \frac{m_n^* m_n^{j-1}}{m^* m^{j-1}} - 1 \right| X^{(n,j)},
\end{align*}
where $X^{(n,j)} = Z^{(n,j)}/(m_n^* m_n^{j-1})$ is a mean one nonnegative martingale conditionally on $\mathscr{F}_n$. By \eqref{eq:MGconvergence}, it only remains to show that $\max_{1 \leq j \leq j_n} | \frac{m_n^* m_n^{j-1}}{m^* m^{j-1}} - 1 | X^{(n,j)} \stackrel{P}{\to} 0$ as $n \to \infty$. To this end note that for $A_{n,\delta}$ defined above and some constant $H = H(\delta) < \infty$,
\begin{align*}
P\left( \max_{1 \leq j \leq j_n} \left| \frac{m_n^* m_n^{j-1}}{m^* m_n^{j-1}} - 1 \right| X^{(j)} > \varepsilon, \, A_{n,\delta} \right) &\leq P\left( H \left( j_n d_1(\nu_n, \nu) + d_1(\nu_n^*, \nu^*) \right) \max_{1\leq j \leq j_n} X^{(n,j)} > \varepsilon, \ A_{n,\delta} \right) \\
&\leq \frac{H}{\varepsilon} E\left[ 1(A_{n,\delta}) \left( j_n d_1(\nu_n, \nu) + d_1(\nu_n^*, \nu^*) \right) E[ X^{(n,j_n)} | \mathscr{F}_n] \right] \\
&=  \frac{H}{\varepsilon} E\left[ 1(A_{n,\delta}) \left( j_n d_1(\nu_n, \nu) + d_1(\nu_n^*, \nu^*) \right) \right] \to 0,
\end{align*}
as $n \to \infty$, where in the second step we used Doob's maximal inequality conditionally on $\mathscr{F}_n$. This completes the proof.
\end{proof}

\bigskip

The last proof verifies the conditions of Theorem \ref{T.MainWBThomo} for the size-biased empirical distribution.

\begin{proof}[Proof of Lemma \ref{L.SizeBiased}]
To analyze $d_1(\nu_n^*, \nu^*)$ define $F(k) = \sum_{i=0}^k f(i)$ and let $F_n(k)$ denote the empirical distribution function of $F$. Then,
\begin{align*}
d_1(\nu_n^*, \nu^*) &= \sum_{k=0}^\infty |F_n(k) - F(k)| = \sum_{k=0}^\infty \left| \frac{1}{n} \sum_{i=1}^n 1(D_i > k) - P(D > k) \right| = \frac{1}{n} \sum_{k=0}^\infty \left| S_n^*(k) \right|,
\end{align*}
where $S_n^*(k) = Y_{k,1} + \dots + Y_{k,n}$ and $Y_{k,i} = 1(D_i > k) - P(D > k)$. Hence, we need to show that
$$\frac{1}{n^{1-\delta'}} \sum_{k=0}^\infty |S_n^*(k)| \stackrel{P}{\longrightarrow} 0,$$
as $n \to \infty$ for any $0 < \delta' < 1/2$. This follows from noting that for any $\varepsilon > 0$,
\begin{align*}
P\left( \sum_{k=0}^\infty |S_n^*(k)| > \varepsilon n^{1-\delta'} \right) &\leq \frac{1}{\varepsilon n^{1-\delta'}} \sum_{k=0}^\infty E[|S_n^*(k)|] \leq \frac{1}{\varepsilon n^{1-\delta'}} \sum_{k=0}^{\infty} \left( E[ (S_n^*(k))^2 ] \right)^{1/2} \\
&= \frac{1}{\varepsilon n^{1-\delta'}} \sum_{k=0}^{\infty} \left( n \var(Y_{k,1}) ] \right)^{1/2} \leq \frac{1}{\varepsilon n^{1/2-\delta'}} \sum_{k=0}^{\infty} \left( P(D > k)   \right)^{1/2} \\
&\leq  \frac{1}{\varepsilon n^{1/2-\delta'}} \left(1 + \sum_{k=1}^{\infty} \left( \frac{E[D^{2+\epsilon}]}{k^{2+\epsilon} }  \right)^{1/2} \right) \to 0,
\end{align*}
as $n \to \infty$, where in the second inequality we used the monotonicity of the $L_p$-norm and in the last one Markov's inequality.

Similarly, if we let $G_n(k) = \sum_{i=1}^n D_i 1(D_i \leq k+1)/L_n$ and $G(k) = E[D 1(D \leq k+1)]/E[D]$, for $k = 0, 1, 2, \dots$, then,
\begin{align*}
d_1(\nu_n, \nu) &= \sum_{k=0}^\infty  |G_n(k) - G(k)| = \sum_{k=0}^\infty \left| \sum_{i=1}^n \frac{D_i}{L_n} 1(D_i > k+1) - \frac{E[D 1(D > k+1)]}{E[D]} \right| \\
&\leq \sum_{k=0}^\infty \left| \sum_{i=1}^n D_i 1(D_i > k+1) \left( \frac{1}{L_n} - \frac{1}{nE[D]} \right)  \right| \\
&\hspace{5mm} + \frac{1}{E[D]} \sum_{k=0}^\infty \left| \frac{1}{n} \sum_{i=1}^n D_i 1(D_i > k+1) - E[D 1(D > k+1)] \right| \\
&\leq \frac{|nE[D] - L_n|}{L_n} \cdot \frac{1}{nE[D]} \sum_{k=0}^\infty \sum_{i=1}^n D_i 1(D_i > k+1) \\
&\hspace{5mm} + \frac{1}{n E[D]} \sum_{k=0}^\infty \left| \sum_{i=1}^n \left(D_i 1(D_i > k+1) - E[D 1(D> k+1)] \right) \right| \\
&= \frac{|nE[D] - L_n|}{n} \cdot \frac{n}{L_n} \cdot \frac{1}{nE[D]} \sum_{i=1}^n D_i (D_i-1)^+  + \frac{1}{nE[D]} \sum_{k=0}^\infty |S_n(k)|,
\end{align*}
where $S_n(k) = X_{k,1} + \dots + X_{k,n}$ and $X_{k,i} = D_i 1(D_i > k+1) - E[D 1(D > k+1)]$. By the weak law of large numbers
$$\frac{n}{L_n} \cdot \frac{1}{nE[D]} \sum_{i=1}^n D_i(D_i-1)^+ \stackrel{P}{\longrightarrow} \frac{E[D (D-1)^+]}{(E[D])^2}$$
as $n \to \infty$, therefore it suffices to show that
$$\frac{|nE[D] - L_n|}{n^{1-\delta}} \stackrel{P}{\longrightarrow} 0 \qquad \text{and} \qquad \frac{1}{n^{1-\delta}} \sum_{k=0}^\infty |S_n(k)| \stackrel{P}{\longrightarrow} 0$$
for $0 < \delta < \min\{1/2, \epsilon/(2+\epsilon)\}$. Since $E[D^{2+\epsilon}] < \infty$, the Marcinkiewicz-Zygmund Strong Law gives that,
$$\frac{(n E[D] - L_n)}{n^{1/p}} = \frac{1}{n^{1/p}} \sum_{i=1}^n (D_i - E[D]) \to 0 \qquad \text{a.s.,}$$
as $n \to \infty$ for any $1 \leq p < 2$, in particular, for $1/p = 1-\delta$. For the limit involving $S_n(k)$ let $a_n = \lfloor n^{1/(2+\epsilon)} \rfloor$ and follow the same steps used to analyze $S_n^*(k)$ to obtain
\begin{align*}
P\left( \sum_{k=0}^\infty |S_n(k)| > \varepsilon n^{1-\delta} \right) &\leq \frac{1}{\varepsilon n^{1-\delta}} \sum_{k=0}^{a_n-1} \left( E[(S_n(k))^2 ] \right)^{1/2} + \frac{1}{\varepsilon n^{1-\delta}} \sum_{k=a_n}^{\infty} E[|S_n(k)| ] \\
&\leq \frac{1}{\varepsilon n^{1/2-\delta}} \sum_{k=0}^{a_n-1} \left( \var(X_{k,1}) \right)^{1/2} + \frac{n^\delta}{\varepsilon} \sum_{k=a_n}^\infty E[|X_{k,1}|] \\
&\leq \frac{1}{\varepsilon n^{1/2-\delta}} \sum_{k=0}^{a_n-1} \left( E[ D^2 1(D > k+1) ] \right)^{1/2} + \frac{n^\delta}{\varepsilon} \sum_{k=a_n}^\infty 2 E[D 1(D > k+1)].
\end{align*}
Using the inequality
\begin{equation} \label{eq:MyMarkov}
E[D^{2+\epsilon}] \geq E[D^{2+\epsilon} 1(D > r)] \geq r^{2+\epsilon-t} E[D^t 1(D > r)],
\end{equation}
for any $r \geq 1$ and any $t \in [0, 2+\epsilon]$, gives that
\begin{align*}
P\left( \sum_{k=0}^\infty |S_n(k)| > \varepsilon n^{1-\delta} \right) &= O\left( \frac{1}{n^{1/2-\delta}} \sum_{k=0}^{a_n-1} \frac{1}{(k+1)^{\epsilon/2}} + n^\delta \sum_{k=a_n}^\infty \frac{1}{(k+1)^{1+\epsilon}}  \right) \\
&= O\left( \frac{a_n^{1-\epsilon/2} 1(\epsilon \neq 2) + (\log a_n) 1(\epsilon = 2)}{n^{1/2-\delta}} + n^\delta a_n^{-\epsilon} \right) \\
&= O\left( n^{\delta - \epsilon/(2+\epsilon)} + n^{\delta - 1/2} (\log n) 1(\epsilon = 2) \right)
\end{align*}
as $n \to \infty$, which converges to zero for $\delta < \min\{1/2, \epsilon/ (2+\epsilon) \}$.
\end{proof}

\section*{Acknowledgements}
We would like to thank an anonymous referee for his or her comments, which were very helpful in improving the presentation and scope of the results.

\bibliographystyle{plain}

\end{document}